\pdfoutput=1
\documentclass{amsart}
\usepackage{amsmath}
\usepackage{amssymb}
\usepackage{graphicx}
\usepackage{hyperref}

\theoremstyle{definition}
\newtheorem{definition}{Definition}
\newtheorem{example}{Example}
\newtheorem{notation}{Notation}
\newtheorem{move}{Move}

\theoremstyle{theorem}
\newtheorem{theorem}{Theorem}
\newtheorem{proposition}{Proposition}
\newtheorem{lemma}{Lemma}
\newtheorem{question}{Question}
\newtheorem{corollary}{Corollary}
\newtheorem{fact}{Fact}
\newtheorem{observe}{Observation}

\theoremstyle{remark}
\newtheorem{remark}{Remark}

\newcommand{\ri}{\operatorname{RI}}
\newcommand{\kk}{\operatorname{knot}}

\title{Crosscap number and knot projections}
\author{Noboru Ito}
\address{Graduate School of Mathematical Sciences, The University of Tokyo, 3-8-1, Komaba, Meguro-ku, Tokyo, 153-8914, Japan}
\email{noboru@ms.u-tokyo.ac.jp}
\author{Yusuke Takimura}
\address{Gakushuin Boys' Junior High School, 1-5-1 Mejiro, Toshima-ku, Tokyo, 171-0031, Japan}
\email{Yusuke.Takimura@gakushuin.ac.jp}
\date{September 27, 2018}
\begin{document}
\begin{abstract}
We introduce an unknotting-type number of knot projections that gives an upper bound of the crosscap number of knots.  
We determine the set of knot projections with the unknotting-type number at most two, and this result implies classical and new results that determine the set of alternating knots with the crosscap number at most two.    
\end{abstract}
\maketitle

\section{Introduction}
In this paper, we introduce an unknotting-type number of knot projections (Definition~\ref{def1})  as follows.  
Every double point in a knot projection can be spliced two different ways (Figure~\ref{f2}), one of which gives another knot projection (Definition~\ref{dfn1}).    
A special case of such operations is a first Reidemeister move $\ri^-$, as shown in Figure~\ref{f2}.  If the other case of such operations, which is not of type $\ri^-$, it is denoted by $S^-$.  
Beginning with an $n$-crossing knot projection $P$, there are many sequences of $n$ splices of type $\ri^-$ and type $S^-$, all of which end with the simple closed curve $O$.    
Then, 
we define the number $u^-(P)$ as the minimum number of splices of type $S^-$ (Definition~\ref{dfn3}).  For this number, we determine the set of knot projections with $u^-(P)=1$ or $u^-(P)=2$ (Theorem~\ref{thm1}, Section~\ref{sec3}).   Here, we  provide an efficient method to obtain a knot projection $P$ with $u^- (P)$ $=$ $n$ for a given $n$ (Move~\ref{lemma2}).  
Further, for a connected sum (Definition~\ref{dfn_connected}) of knot projections, we show that  the additivity of $u^-$ under the connected sum (Section~\ref{sec7}).     
Thus, to calculate $u^-(P)$ for every knot projection $P$, it is sufficient to compute $u^-(P)$ for every prime knot projection $P$.  Here, a knot projection is called a \emph{prime} (Definition~\ref{dfn_connected}) knot projection if the knot projection is not the simple closed curve and is not a connected sum of two knot projections, each of which is not the simple closed curve.  

We apply this unknotting-type number to the theory of crosscap numbers (Section\ref{sec4}--Section~\ref{sec6}).  Let $P$ be a knot projection and $D_P$ a knot diagram by adding any over/under information to each double point of $P$.  Let $K(D_P)$ be a knot type (Definition~\ref{def1}) where $D_P$ is a representative of the knot type.   In particular, if $D_P$ is an alternating knot diagram (Definition~\ref{def1}), we  denote $K(D_P)$ by $K^{alt}(P)$ simply.   In this paper, in general, we show that the unknotting-type number $u^-(P)$ gives an upper bound of the crosscap number of knots (Definition~\ref{dfn_cross}), i.e., $C(K(D_P)) \le u^-(P)$ (Theorem~\ref{thm2}, Section~\ref{sec4}).    As a special case if $K(D_P)=K^{alt}(P)$, as a corollary of the inequality, we are easily able to determine the set of alternating knots with $C(K)=1$ (corresponding to a classical result in Section~\ref{sec5}) or $C(K)=2$ (corresponding to a new result in Section~\ref{sec6}).   Similarly, by using type $S^+$ ($\ri^+$,~resp.) that is the inverse operation of type $S^-$ ($\ri^-$,~resp.), we also introduce $u(P)$ that is the minimum number of operations of types $S^{\pm}$ in a sequence, from $P$ to $O$, consisting of operations of types $S^{\pm}$ and $\ri^{\pm}$.  
These studies are motivated by Observation~\ref{ob1}, where we use crosscap numbers in the table of KnotInfo \cite{CL} and a table of knot projections up to eight double points \cite{IT8c}.    
\begin{observe}\label{ob1}
For every prime knot projection $P$ with less than nine double points,   
$C(K^{alt} (P))$ $=$ $u(P)$ $=$ $u^-(P)$.  
\end{observe}
In Section~\ref{sec7}, for a connected sum $P \sharp P'$, we give examples of $u(P \sharp P')$ and $u^- (P \sharp P')$ satisfying $u(P \sharp P')$ $<$ $u^- (P \sharp P')$.   We also obtain a question whether every knot projection $P$ holds $C(K(D_P))$ $\le u(P)$.  

Finally, we would like to mention that crosscap numbers of knots are discussed in the literature.  Clark obtained that for a knot $K$, $C(K)=1$ if and only if $K$ is a $2$-cable knot (in particular, for an alternating knot $K$, $K$ is a $(2, p)$-torus knot) \cite{clark}.  Clark also obtained an upper bound $C(K) \le 2 g(K)+1$, where $g(K)$ is the orientable genus of $K$ (this inequality holds for every knot $K$) \cite{clark}.   Murakami and Yasuhara \cite{MY} gave the example $C(K)=2g(K)+1$ by $K = 7_4$ and sharp bounds $C(K) \le \lfloor n(K)/2 \rfloor$ for the minimum crossing number $n(K)$ of a knot $K$ (note that, as we mention in the following, Hatcher-Thurston \cite{HaT} includes the particular case $7_4$, which is discussed \emph{explicitly} in Hirasawa-Teragaito \cite{HT}).  Murakami and Yasuhara \cite{MY} also gave the necessary and sufficient condition for the crosscap number to be additive under the connected sum.   
Historically, the orientable knot genus has been well studied, and a general algorithm for computations is known.  For low crossing number knots, effective calculations are made from genus bounds using invariants such as the Alexander polynomial and the Heegaared Floer homology.  

However, crosscap numbers are harder to compute.   In this situation, crosscap numbers of several  families are known by Teragaito (torus knots) \cite{Tra}, Hatcher-Thurston ($2$-bridge knots, in theory), Hirasawa-Teragaito ($2$-bridge knots, explicitly) \cite{HT}, Ichihara-Mizushima (\emph{many} pretzel knots) \cite{IM}.   
Adams and Kindred \cite{AK} determine the crosscap number of any alternating knot in theory.  For a given $n$ crossing alternating knot diagram, consider $2^n - 1$ non-orientable state surfaces (Definition~\ref{state}); some of these surfaces achieve the crosscap number of the knot.   
By using coefficients of the colored Jones polynomials to establish two sided bounds on crosscap numbers, Kalfagianni and Lee \cite{KL} improve the efficiency of these computations.  They apply this improved efficiency in order to calculate hundreds of crosscap numbers explicitly and rapidly.    
 
In this paper, we relate our unknotting-type number $u^-(P)$ of a knot projection $P$ to methods of Adams-Kindered \cite{AK} that seems at a glance to be distinct from giving our number $u^-(P)$.    We also study crosscap numbers from a different viewpoint to obtain a state surface for an alternating knot using our unknotting-type number $u^-(P)$ of a knot projection $P$, and determine the set of alternating knots with the crosscap number two.

\section{Preliminaries}
\begin{definition}[knot, knot diagram, knot projection, alternating knot]\label{def1}
A \emph{knot} is an embedding from a circle to $\mathbb{R}^3$.     We say that knots $K$ and $K'$ are \emph{equivalent} if there is a homeomorphism of $\mathbb{R}^3$ onto itself which maps $K$ onto $K'$.  Then, each equivalence class of knots is called a \emph{knot type}.  
A \emph{knot projection} is an image of a generic immersion from a circle into $S^2$ where every singularity is a transverse double point.     
In this paper, a transverse double point of a knot projection is simply called a \emph{double point}.   
The \emph{simple closed curve} is a knot projection with no double points.  
A \emph{knot diagram} is a knot projection with over/under information for every double point.   
Throughout this paper, in general, a knot projection (knot diagram, resp.) is defined not to be distinct from its mirror image.
A double point with over/under information of a knot diagram is called a \emph{crossing}.  As a special case, for a knot projection, one can arrange crossings in such a way that an under-path and an over-path alternate when traveling along the knot projection.   Then,  the knot diagram is called an \emph{alternating knot diagram}.  For a knot $K$, if a knot diagram of $K$ is an alternating knot diagram, then $K$ is called an \emph{alternating knot}.  
\end{definition}
\begin{definition}[splices, operations of type $S^-$ or type $\ri^-$, Seifert splice]\label{dfn1}
For each double point, there are two ways to smooth the knot projection near the double point (Figure~\ref{f1}~(a) $N(d)$).  Namely, erase the transversal intersection of the knot projection within a small neighborhood of the double point and connect the four resulting ends by a pair of simple, nonintersecting arcs (Figure~\ref{f1}~(b) or (c) $N'(d)$).       
A replacement of $N(d)$ with $N'(d)$ is called a \emph{splice}.  
\begin{figure}[h!]
\includegraphics[width=6cm]{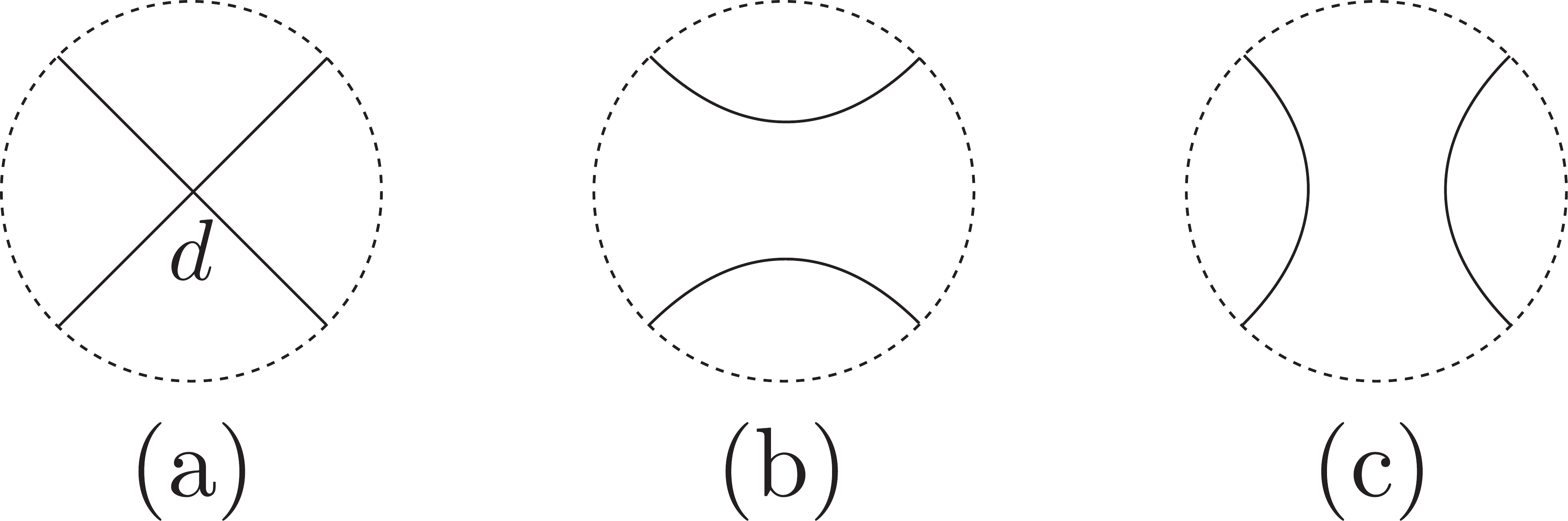}
\caption{(a) : $N(d)$, (b) : $N'(d)$, and (c) : another $N'(d)$}\label{f1}
\end{figure} 
If a connection of four points in $S^2 \setminus N(d)$ is fixed, the connection is presented by dotted arcs as in Figure~\ref{f2}~(a-1) or Figure~\ref{f2}~(c-1) (ignoring the orientation).  First, we consider a splice from (a-1) to (a-2) in Figure~\ref{f2}.   Then, a special case of such splices as in (b-1) to (b-2) is called a \emph{splice of type} $\ri^-$ and is denoted by $\ri^-$.  If the case is not $\ri^-$, as in (b-1) to (b-2), then the operation is called a \emph{splice of type} $S^-$ and is denoted by $S^-$.  Second, if we choose a connection presented by dotted arcs as in Figure~\ref{f2} (c-1), the splice from (c-1) to (c-2) in Figure~\ref{f2} is called a \emph{Seifert splice} or a splice of type \emph{Seifert}.  The splice preserves the orientation of the knot projection, as in Figure~\ref{f2}.  
\end{definition}
By definition, we have Fact~\ref{fact0}.  
\begin{fact}\label{fact0}
Every splice is one of three types: $S^-$, $\ri^-$, or Seifert.  
\end{fact}
\begin{figure}[h!]
\includegraphics[width=12cm]{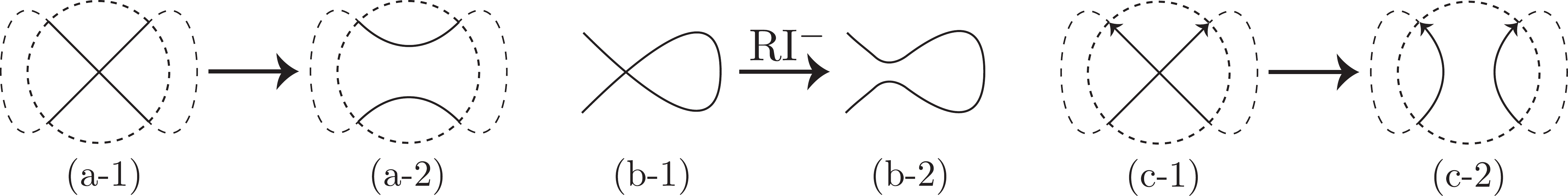}
\caption{Three types of splices are represented by pairs ((a-1), (a-2)), ((b-1), (b-2)), and ((c-1), (c-2)).    
In the operation from (c-1) to (c-2), the orientation is ignored in Definition~\ref{dfn1} and it is not ignored in Definition~\ref{state2}.}\label{f2}
\end{figure}
\begin{remark}
An operation here is introduced by \cite{Ca} (for the full twisted version) and \cite{ItoShimizu} (for the half-twisted version), and in \cite{ItoShimizu}, it is called the inverse of a \emph{half-twisted splice operation}, denoted by $A^{-1}$.   
\end{remark}
By definition, it is easy to see Fact~\ref{fact1} (it is a known fact).  
\begin{fact}\label{fact1}
Let $P$ be a knot projection with $n$ double points.  
There exist at most $2^n$ distinct sequences of splices of type $S^-$ and $\ri^-$ from $P$ to the simple closed curve $O$.  Each sequence consists of $n$ splices in total.       
\end{fact}
\begin{definition}[unknotting-type number $u^-(P)$]\label{dfn3}
Let $P$ be a knot projection and $O$ the simple closed curve.  The nonnegative integer $u^- (P)$ is defined as the minimum number of splices of type $S^-$ for any finite sequence of splices of type $S^-$ and of type $\ri^-$ to obtain $O$ from $P$.  
\end{definition}
\begin{example}
Figure~\ref{ex0} gives examples of knot projections with $u^-(\widehat{1_1})$ $=$ $0$, $u^-(\widehat{3_1})$ $=$ $1$, or $u^-(\widehat{6_2})$ $=$ $2$.  Here, letting $i$ be a positive integer, for a knot diagram $n_i$ in the famous table in \cite{Ro}, the corresponding knot projection is denoted by $\widehat{D}$ (for details, see \cite{IT8c}).   
\end{example}
\begin{figure}[h!]
\includegraphics[width=12cm]{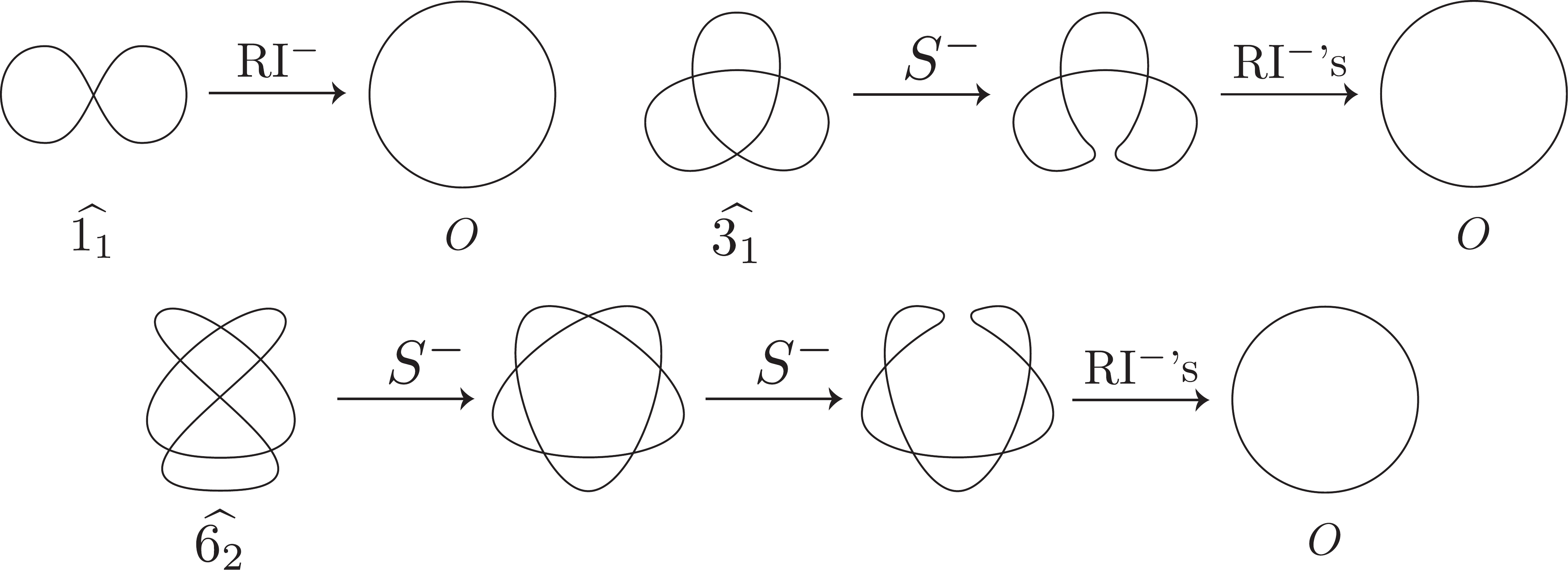}
\caption{$u^-(\widehat{1_1})$ $=$ $0$, $u^-(\widehat{3_1})$ $=$ $1$, or $u^-(\widehat{6_2})$ $=$ $2$.}\label{ex0}
\end{figure}
The definition of a connected sum of two knots is slightly different from that of two knot projections, which is not unique, as in Definition~\ref{dfn_connected}.     
\begin{definition}[a connected sum of two knot projections, a prime knot projection]\label{dfn_connected}
Let $P_i$ be a knot projection ($i=1, 2$).  Suppose that the ambient $2$-spheres corresponding to $P_1, P_2$ are oriented.   Let $p_i$ be a point on $P_i$ where $p_i$ is not a double point ($i=1, 2$).  Let $d_i$ be a sufficiently small disk with the center $p_i$ ($i=1, 2$) satisfying $d_i \cap P_i$ consists of an arc which is properly embedded in $d_i$.  Let $\widetilde{d}_i$ $=$ $cl(S^2 \setminus d_i)$, $\widetilde{P}_i$ $=$ $P_i \cap \widetilde{d}_i$, and let $h :$ $\partial \widetilde{d}_1$ $\to$ $\partial \widetilde{d}_2$  be an orientation reversing homeomorphism where $h(\partial \widetilde{P}_1)$ $=$ $\partial \widetilde{P}_2$.   Then $\widetilde{P}_1 \cup_h \widetilde{P}_2$ gives a knot projection in the oriented $2$-sphere $\widetilde{d}_1 \cup_h \widetilde{d}_2$.   The knot projection $\widetilde{P}_1 \cup_h \widetilde{P}_2$ in the oriented $2$-sphere is denoted by $P_1 \sharp_{(p_1,~ p_2,~h)} P_2$ and is called a \emph{connected sum} of the knot projections $P_1$ and $P_2$ at the pair of points $p_1$ and $p_2$ (Figure~\ref{f3}).  A connected sum of knot projections is often simply denoted by $P_1 \sharp P_2$ when no confusion is likely to arise.  
If a knot projection is not the simple closed curve and is not a connected sum of two knot projections, each of which is not the simple closed curve, it is called a \emph{prime knot projection}.  
\end{definition}   
\begin{figure}[h!]
\includegraphics[width=12cm]{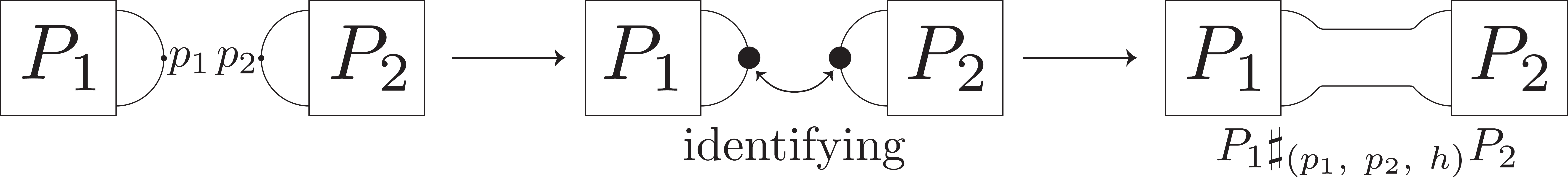}
\caption{A connected sum $P_1 \sharp_{(p_1,~p_2,~h)} P_2$ of two knot projections $P_1$ and $P_2$}\label{f3}
\end{figure}
\begin{definition}[the connected sum of two knots]
Let $K_i$ be a knot ($i=1, 2$) and $D_i$ a knot diagram of $K_i$.   Let $P_i$ be a knot projection corresponding to $D_i$.   A connected sum $D_1 \sharp_{(p_1,~p_2,~h)} D_2$ is defined as a connected sum $P_1 \sharp_{(p_1,~p_2,~h)} P_2$ in Definition~\ref{dfn_connected}.   Then, a knot having a knot diagram $D_1 \sharp_{(p_1,~p_2,~h)} D_2$ is called a connected sum of $K_1$ and $K_2$.   Because it is well-known that a connected sum of $K_1$ and $K_2$ does not depend on $(p_1,~p_2,~h)$, the connected sum is denoted by $K_1 \sharp K_2$.   
\end{definition}
\begin{definition}[crosscap number]\label{dfn_cross}
The \emph{crosscap number} $C(K)$ of a knot $K$ is defined by 
$C(K)$ $=$ $\min \{$ $1-\chi(\Sigma)~|~$ a non-orientable surface $\Sigma$ with $\partial \Sigma = K \}$, where $\chi(\Sigma)$ is the Euler characteristic of $\Sigma$.  
Traditionally, we define that $K$ is the unknot if and only if $C(K)=0$.  
\end{definition}
\begin{definition}[set $\langle {\mathcal{S}} \rangle$]\label{ri_eq_notation}
Let $\ri^+$ be the inverse of a splice of type $\ri^-$ (Figure~\ref{f5}).  Let $P$ and $P'$ be knot projections.  
We say that $P \sim P'$ if $P$ and $P'$ are related by a finite sequence of operations of types ${\ri^{\pm}}$.  It is easy to see that $\sim$ defines an equivalence relation.   Let $\mathcal{S}$ be a set of knot projections.  
Let $\langle {\mathcal{S}} \rangle$ $=$ $\{ P~|$ $P \sim Q$ $(\exists Q \in {\mathcal{S}})\}$.    
\end{definition}
\begin{figure}[h!]
\includegraphics[width=5cm]{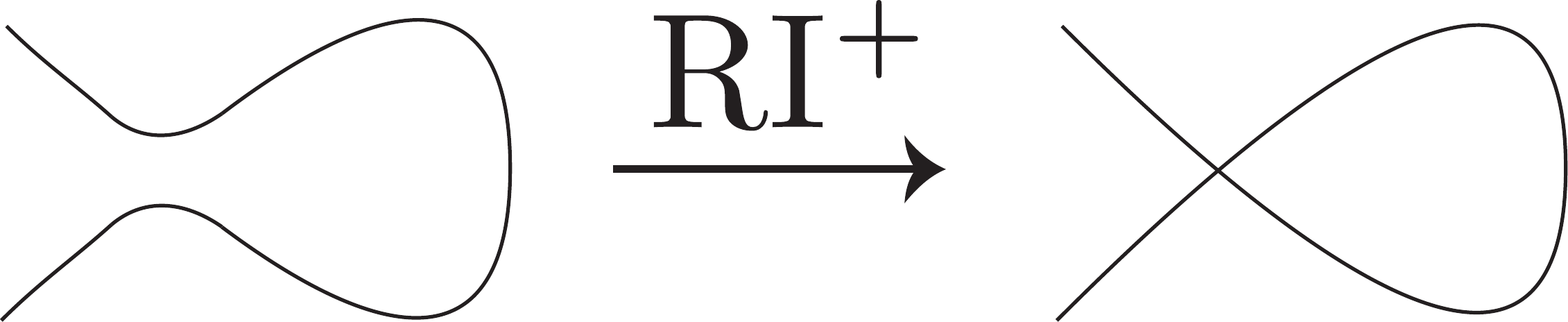}
\caption{$\ri^+$}\label{f5}
\end{figure}
\begin{notation}[Sets $\mathcal{T}$, $\mathcal{R}$, $\mathcal{P}$]\label{not1}
Let $l$, $m$, $n$, $p$, $q$, and $r$ be positive integers.  
Let $\mathcal{T}$ be the set of $(2, 2l-1)$-torus knot projections $(l \ge 2)$, $\mathcal{R}$ the set of $(2m, 2n-1)$-rational knot projections $(m \ge 1, n \ge 2)$, and  $\mathcal{P}$ the set of $(2p, 2q-1, 2r-1)$-pretzel knot projections $(p, q, r \ge 1)$ as in Figure~\ref{f6}.   Let $\langle {\mathcal{T}} \rangle \sharp \langle {\mathcal{T}} \rangle$ $=$ $\{P_1 \sharp P_2 ~|~P_1, P_2 \in  \langle {\mathcal{T}} \rangle \}$.  
\end{notation}
\begin{figure}[h!]
\includegraphics[width=12cm]{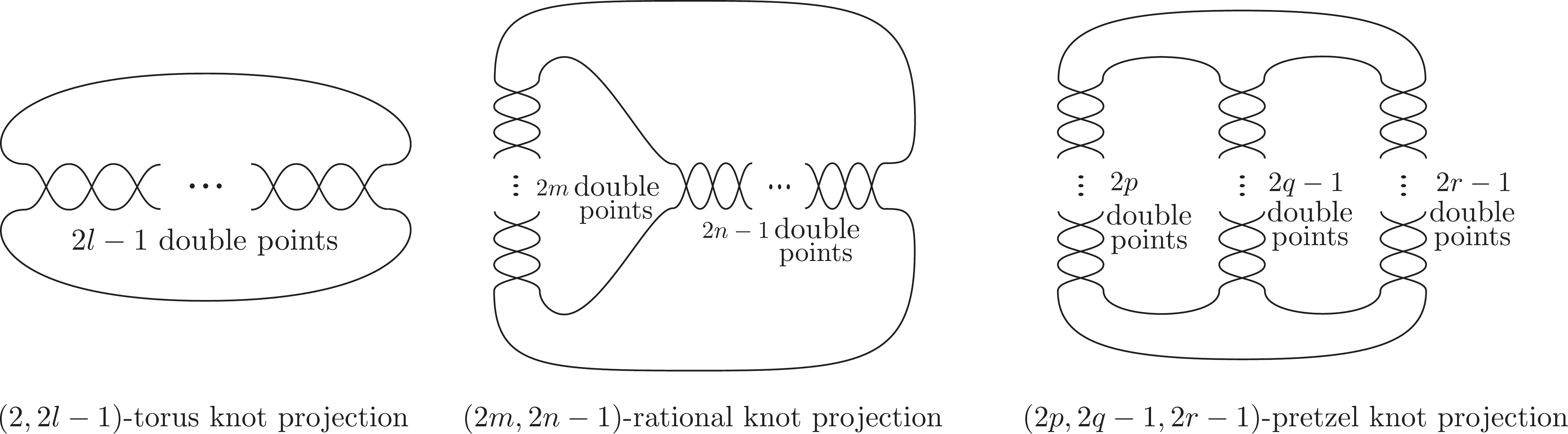}
\caption{A $(2, 2l-1)$-torus knot projection $(l \ge 2)$, a $(2m, 2n-1)$-rational knot projection $(m \ge 1, n \ge 2)$, and  a $(2p, 2q-1, 2r-1)$-pretzel knot projection $(p, q, r \ge 1)$}\label{f6}
\end{figure}
\begin{notation}\label{not3}
Let $l$, $m$, $n$, $p$, $q$, and $r$ be positive integers.  
Let $\mathcal{T}_{\kk}$ ($\mathcal{R}_{\kk}$, $\mathcal{P}_{\kk}$, resp.) be the set of $(2, 2l-1)$-torus knots $(l \ge 2)$ ($(2m, 2n-1)$-rational knots $(m \ge 1, n \ge 2)$,  $(2p, 2q-1, 2r-1)$- pretzel knots $(p, q, r \ge 1)$,~resp.) as in Figure~\ref{f8}.  Let $\mathcal{T}_{\kk} \sharp \mathcal{T}_{\kk}$ $=$ $\{ L \sharp L' ~|~L, L' \in {\mathcal{T}_{\kk}} \}$.  

\begin{figure}[h!]
\includegraphics[width=12cm]{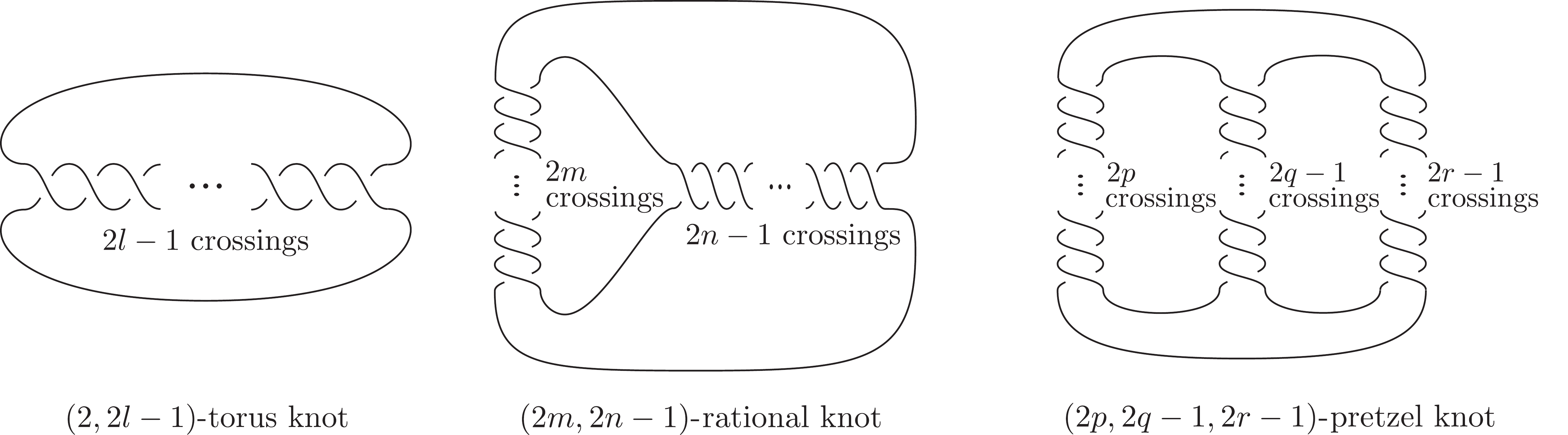}
\caption{Knot diagrams of knots in $\mathcal{T}_{\kk}$, $\mathcal{R}_{\kk}$, and $\mathcal{P}_{\kk}$ $(l \ge 2, m \ge 1, n \ge 2, p, q, r \ge 1)$
}\label{f8}
\end{figure}  
\end{notation}
\begin{definition}\label{def_uk}
Let $K$ be an alternating knot and $C(K)$ the crosscap number of $K$.  Let $Z(K)$ be the set of knot projections obtained from alternating knot diagrams of $K$.  Then, $\min_{P \in Z(K)} u^-(P)$ is an alternating knot invariant.  Let $u^- (K)$ $=$ $\min_{P \in Z(K)} u^-(P)$.  
\end{definition}

\section{Knot projections with $u^- (P)$ $\le$ $2$}\label{sec3}
\begin{theorem}\label{thm1}
Let $P$ be a knot projection.      
Let $\mathcal{T}$, $\mathcal{R}$, and  $\mathcal{P}$ be the sets as in Notation~\ref{not1}.  Then,   
\begin{enumerate}
\item $u^-(P)$  $=$ $1$ if and only if $P$ $\in \langle \mathcal{T} \rangle$.  
\item $u^-(P)$  $=$ $2$ if and only if $P$ $\in \langle \mathcal{R} \rangle \cup \langle {\mathcal{P}} \rangle \cup \langle {\mathcal{T}} \rangle \sharp \langle {\mathcal{T}} \rangle$.  
\end{enumerate}
\end{theorem}
In the proof of Theorem~\ref{thm1}, we prepare Notation~\ref{notation_p} and a move (Move~\ref{lemma2}).   
\begin{notation}\label{notation_p}
Let $S^+$ be the inverse operation of a splice of type $S^-$.    
\end{notation}
\begin{move}\label{lemma2}
For any pair of simple arcs lying on the boundary of a common region, each of the two local replacements as in Figure~\ref{p1} is obtained by applying operations of type {$\ri^+$} $i-1$ times followed by a single operation of type $S^+$.     
\end{move}
\begin{figure}[h!]
\includegraphics[width=12cm]{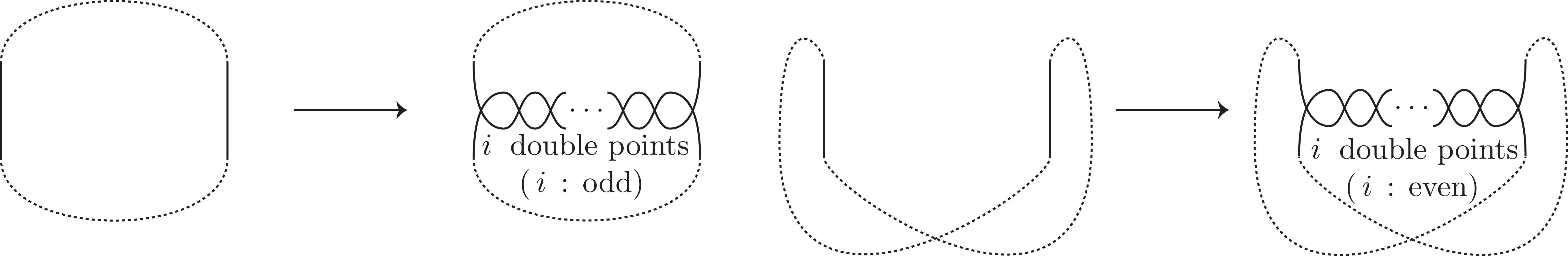}
\caption{Two local replacements}\label{p1}
\end{figure}
For the proof of Theorem~\ref{thm1}, 
it is worth to mention Proposition~\ref{prop_move}. 
\begin{proposition}\label{prop_move}
Let $P$ be a knot projection and $O$ the simple closed curve.  Let $\langle \{O\} \rangle$ be as in Definition~\ref{ri_eq_notation}.  
The following conditions are equivalent.  
\begin{enumerate}
\item[(A)] $P$ satisfies $u^- (P)=n$.  
\item[(B)] There exists $Q \in \langle \{O\} \rangle$ such that $P$ is obtained from $Q$ by applying Move~\ref{lemma2} successively $n$ times.  
\end{enumerate}
\end{proposition}
Now we prove Theorem~\ref{thm1} in the following.
\begin{proof}
\noindent (1).  
For the simple closed curve $O$, if we apply a finite sequence of a single $S^+$ and {$\ri^+$}'s, which corresponds to Move~\ref{lemma2}, then $P \in {\mathcal{T}}$.  If some ${\ri^+}$'s are applied to $P$, we have $P' \in \langle  {\mathcal{T}} \rangle$.  

Conversely, suppose that $P' \in \langle  {\mathcal{T}} \rangle$.  Then $P (\in {\mathcal{T}})$ is obtained from $P'$ by some applications of {$\ri^-$}'s.  For $P (\in {\mathcal{T}})$, it is easy to find a single $S^-$ to obtain an element of $\langle \{ O \} \rangle$.  

\noindent (2).   By (1), the argument starts on $P \in {\mathcal{T}}$.   
For each of three marked places, $\alpha$, $\beta$, and $\gamma$ as in Figure~\ref{p2a}, we find a pair of simple arcs lying on the boundary of a common region.   By applying Move~\ref{lemma2} to $\alpha$ ($\beta$, $\gamma$, resp.), we have $P \in {\mathcal{R}}$ ($P \in {\mathcal{P}}$, $P \in \langle {\mathcal{T}} \rangle \sharp \langle {\mathcal{T}} \rangle$,~resp.).  
Note that for $\gamma$, there is the ambiguity to apply Move~\ref{lemma2}.   However, essentially, the same argument works.  See Figure~\ref{p2b}.  

Conversely, for $P$ $\in  \langle \mathcal{R} \rangle \cup \langle {\mathcal{P}} \rangle \cup \langle {\mathcal{T}} \rangle \sharp \langle {\mathcal{T}} \rangle$, it is easy to find a single $S^-$ to obtain an element of $\langle {\mathcal{T}} \rangle$.
\end{proof}
\begin{figure}[h!]
\includegraphics[width=4cm]{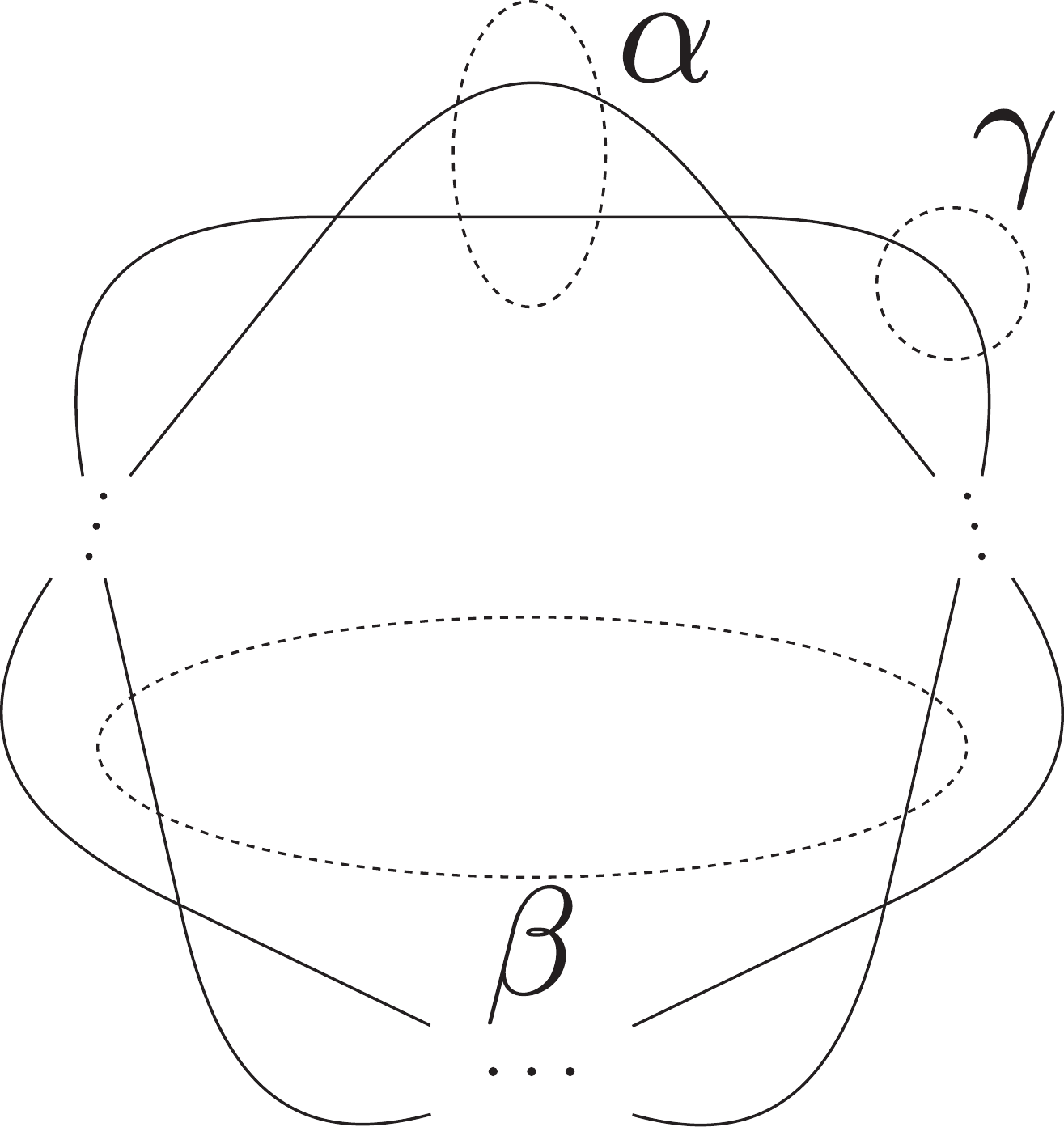}
\caption{Three places for an application of one of the local replacement}\label{p2a}
\end{figure}
\begin{figure}[h!]
\includegraphics[width=10cm]{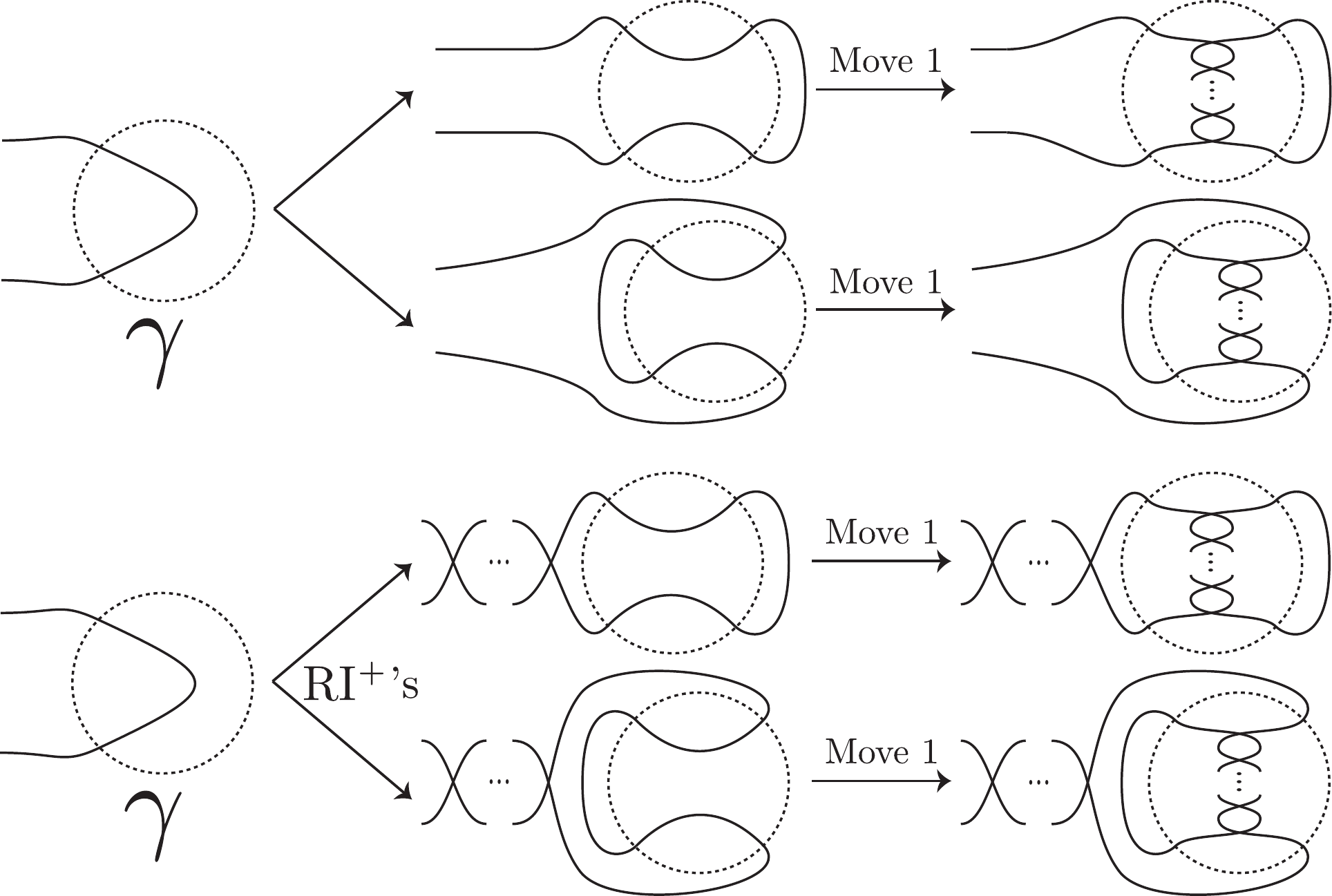}
\caption{Move~\ref{lemma2} and $\gamma$}\label{p2b}
\end{figure}

\section{$u^- (P)$ and crosscap numbers}\label{sec4}
\begin{definition}[state surface, cf.~\cite{AK}]\label{state}
Let $P$ be a knot projection and $D_P$ a knot diagram by adding any over/under information to each double point of $P$.  Let $K(D_P)$ be a knot type where $D_P$ is a representative of the knot type.    
By using the identification $S^2$ $=$ $R^2 \cup \{ \infty \}$, a knot projection $P$ (knot diagram $D_P$, resp.) is considered on $\mathbb{R}^2$ in the following.   
For a knot projection, by applying a splice to each double point, we have an arrangement of disjoint circles on $\mathbb{R}^2$.  The resulting arrangement of circles on $\mathbb{R}^2$ are called a \emph{state} and circles in a state are called \emph{state circles} (cf.~\cite{K}).     For the state, every circle is filled with disks, and the nested disks stacked in some order.  Then the surface is given by attaching half-twisted bands across the crossings of $D_P$ to obtain a surface spanning the knot $K(D_P)$.  The twisting is fixed by the type of the crossing.  The surface generated by this algorithm is called a \emph{state surface}.      

Suppose that a state $\sigma$ of a knot projection $P$ with exactly $n$ double points is given by ordered $n(P)$ splices.    Then, we denote $\sigma$ by  $(\sigma_1, \sigma_2, \dots, \sigma_{n(P)})$ and the resulting arrangement of simple closed curves on a sphere is denoted by $S_\sigma$.   Let $|S_{\sigma}|$ be the number of circles in $S_{\sigma}$.   For a state $\sigma$ of a knot diagram $D_P$, let $\Sigma_{\sigma} (D_P)$ be the state surface obtained from circles and half-twisted bands corresponding to $S_{\sigma}$ and $\sigma$.   
If $\sigma$ satisfies that every $\sigma_i$ is a Seifert splice, the state surface is orientable, it is called a \emph{Seifert state surface}, and the state is called a \emph{Seifert state} (see Fact~\ref{w_fact}).    
\end{definition}
\begin{fact}[a well-known fact]\label{w_fact}
For a positive integer $n$, $2^n$ states from an $n$ crossing knot diagram, all except the Seifert state give non-orientable state surfaces.  
For every alternating knot $K$, there exists a Seifert state surface whose genus is $g(K)$ via an algorithm as in Definition~\ref{state2}.  
\end{fact}  
\begin{definition}[Seifert's algorithm]\label{state2}
For a given knot, we orient it.  Then,  
for every crossing of a knot diagram of the knot, if we choose the splice from (c-1) to (c-2) as in Figure~\ref{f2}, then the state surface given by Definition~\ref{state} is orientable.  
The resulting surface does not depend on the orientation of the knot.  
Traditionally, the process is called \emph{Seifert's algorithm}.     
State circles appearing in the process of Seifert's algorithm are called \emph{Seifert circles}.  
\end{definition}
\begin{theorem}\label{thm2}
Let $P$ be a knot projection and $D_P$ a knot diagram by adding any over/under information to each double point of $P$.  Let $K(D_P)$ be the knot type having a knot diagram $D_P$.  
Let $C(K(D_P))$ be the crosscap number of $K(D_P)$.  
Then, 
\[  
C(K(D_P)) \le u^- (P).  
\]
\end{theorem}
\begin{proof}
Let $n(P)$ be the number of double points of $P$.  In the following, we obtain an appropriate state in the $2^{n(P)}$ candidates to find a state surface by using a sequence realizing $u^-(P)$.  

Consider a sequence of splices that realizes $u^-(P)$.  Denote it by 
\[
P=P_1 \stackrel{Op_1}{\to} P_2  \stackrel{Op_2}{\to} \dots \stackrel{Op_{n(P)}}{\to} O.  
\]
Then, let $\sigma$ $=$ $(\sigma_1, \sigma_2, \dots, \sigma_{n(P)})$ by assigning a splice $\sigma_i$ to each double point of $P$ as follows.    
\begin{itemize}
\item If $Op_i = S^-$ at a double point of $d$, the splice $\sigma_i$ is defined as $S^-$ (Figure~\ref{f1_proof}, the left half).  
\item If $Op_i = \ri^-$ at a double point of $d$, the  splice $\sigma_i$ is defined as the splice which is different from $\ri^-$ (Figure~\ref{f1_proof}, the right half).  
\end{itemize}
\begin{figure}[h!]
\includegraphics[width=10cm]{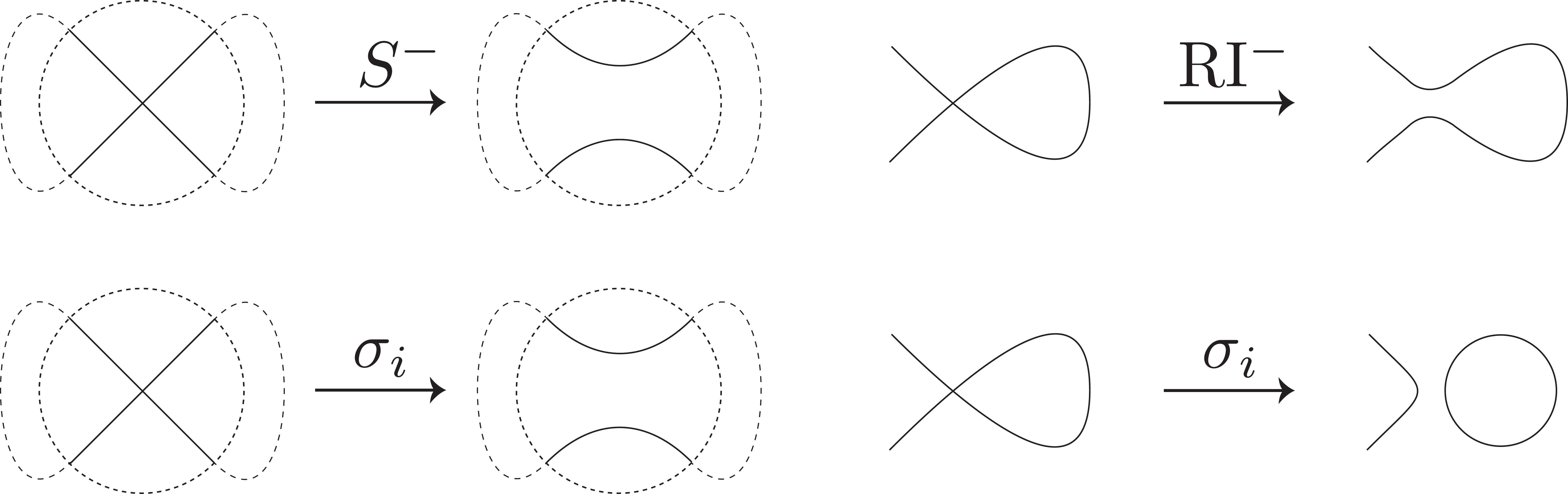}
\caption{$\sigma_i$ for $S^-$ (the left half) and $\sigma_i$ for $\ri^-$ (the right half)}\label{f1_proof}
\end{figure}   
If every $Op_i$ is type $\ri^-$, in which case $K(D_P)$ is the unknot, $C(K(D_P))$ $=$ $0$ and  $u^- (P)$ $=$ $0$, which is one of the case of the statement.  Thus, we may suppose that at least some $Op_i$ is type $S^-$.  Then, since $\sigma$ is not a Seifert state (Definition~\ref{state}), $\Sigma_{\sigma} (D_P)$ is non-orientable (cf.~Fact~\ref{w_fact}).  
 
For $K(D_P)$, let $\Sigma_0$ be a non-orientable surface that spans $K(D_P)$ and satisfies $\chi(\Sigma_0)$ $=$ $1-C(K(D_P))$.   
By the maximality of $\chi(\Sigma_0)$, 
\begin{align}\label{eq0}
\chi(\Sigma_0) \ge \chi(\Sigma_{\sigma} (D_P)).   
\end{align}
Therefore, 
\begin{align}\label{eq1}
1 - C(K(D_P)) &= \chi(\Sigma_0) \ge \chi(\Sigma_{\sigma} (D_P)) = |S_{\sigma}| - n(P).
\end{align}
Note that a splice $\sigma_i$ corresponding to $S^-$ from $P_i$ to $P_{i+1}$ does not change the number of the components and a splice $\sigma_i$ corresponding to $\ri^-$ from $P_i$ to $P_{i+1}$ increases the number of the components by exactly one (Figure~\ref{f1_proof}).  Observing the process in the finite sequence from $P$ to the simple closed curve $O$, it is easy to see $|S_{\sigma}|$ $=$ $1+$ $\sharp \{Op_i~|~ Op_i = \ri^- \}$.  
Note also that $n(P)$ $=$ $\sharp \{Op_i~|~ Op_i = \ri^- \}$ $+$ $\sharp \{Op_j~|~ Op_j = S^- \}$.  
Therefore, 
\[
|S_{\sigma}| - n(P) = 1 + \sharp \{Op_i~|~ Op_i = \ri^- \} - (\sharp \{Op_i~|~ Op_i = \ri^- \} + \sharp \{Op_j~|~ Op_j = S^- \}).  
\]
Thus, 
\begin{equation}\label{eq2}
\begin{split}
1 - C(K(D_P)) &\ge |S_{\sigma}| - n(P) \\
&= 1 -  \sharp \{Op_j~|~ Op_j = S^- \}  \\
& = 1 - u^-(P).  
\end{split}
\end{equation}
\end{proof}
\begin{notation}\label{not4}
For a knot diagram $D_P$ of a knot projection $P$, a particular state surface introduced in the proof of Theorem~\ref{thm2} is denoted by $\Sigma_u$ (it is a state surface corresponding to a sequence of splices that realized $u^- (P)$).      
\end{notation}
By this proof, for the equalities on (\ref{eq0}),  (\ref{eq1}), and (\ref{eq2}), we have Lemma~\ref{lemma0}.     
\begin{lemma}\label{lemma0}
Let $P$, $D_P$, $K(D_P)$, $C(K(D_P))$, and $\Sigma_0$ be as in Theorem~\ref{thm2}, i.e.,    
let $P$ be a knot projection, $D_P$ a knot diagram by adding any over/under information to each double point of $P$, $K(D_P)$ the knot type having a knot diagram $D_P$, $C(K(D_P))$ the crosscap number of $K(D_P)$, $\Sigma_0$ a non-orientable surface that spans $K(D_P)$ and satisfies $\chi(\Sigma_0)$ $=$ $1-C(K(D_P))$.   Let $\Sigma_u$  be as in Notation~\ref{not4}.  

Then, $\chi(\Sigma_0) = \chi(\Sigma_u)$ if and only if $C(K(D_P))=u^-(P)$.  
\end{lemma}
\begin{proof}
By Notation~\ref{not4}, we choose $\Sigma_{\sigma} (D_P)$ that is $\Sigma_u$ as in the proof of Theorem~\ref{thm2}.  Then,     
\begin{align*}
&\chi(\Sigma_0) = \chi(\Sigma_{\sigma} (D_P))~{\textrm{on}}~(\ref{eq0})\\
\Longleftrightarrow& 
 1 - C(K(D_P)) = |S_{\sigma}| - n(P)~{\textrm{on}}~(\ref{eq1}) \\ 
\Longleftrightarrow&
C(K(D_P))=u^-(P)~{\textrm{on}}~(\ref{eq2}).  
\end{align*} 
\end{proof}
To discuss the equality of (\ref{eq0}), we review  Fact~\ref{AKthm}.   Here, we give Definition~\ref{gon} only, and review their fact.  
\begin{definition}[$n$-gon]\label{gon}
Let $P$ be a knot projection and let $\partial F$ be the boundary of the closure of a connected component $F$ of $S^2 \setminus P$.  Let $n$ be a positive integer.  Then, $\partial F$ is called an \emph{$n$-gon} if, when the double points of $P$ that lie on $\partial F$ are removed, the remainder consists of $n$ connected components, each of which is homeomorphic to an open interval.  For a knot diagram, the definition of an $n$-gon is straightforward.      
\end{definition}
Following \cite{AK}, a \emph{genus} is defined to be the orientable genus of a knot or $\frac{1}{2}$ of the crosscap number.  
\begin{fact}[Adams-Kindred, Theorem~3.3 of \cite{AK}]\label{AKthm}
For every alternating knot diagram, the following algorithm $(1)$--$(3)$ always generates a minimal genus state surface.  
\end{fact}
\noindent{\bf{Minimal genus algorithm.}}  

Let $D_P$ be an alternating knot diagram.  

\begin{itemize}
\item[(1)]  Find the smallest $m$ for which $D_P$ contains an $m$-gon.  
\item[(2)] If $m \le 2$, then we apply the splice(s) to the crossing(s) so that the $m$-gon becomes a state circle.  If $m > 2$, then $m=3$ by a simple Euler characteristic argument on the knot projection (see, e.g., \cite[Lemma~3.1]{KL} or \cite[Lemma~2]{IT_triple1}).  Then, choose a triangle of $D_P$.  From here, the process has two branches: For one branch, we apply splices to the crossings on this triangle's boundary so that the triangle becomes a state circle.  For the other branch, we apply splices to the crossings the opposite way.  
\item[(3)] Repeat Steps (1) and (2) until each branch reaches a state.  Of all resulting state surfaces, choose the one with the smallest genus.  
\end{itemize}

Here, recall notations $\langle {\mathcal{T}} \rangle$, $\langle {\mathcal{P}} \rangle$, and $\langle {\mathcal{R}} \rangle$  in Theorem~\ref{thm1} (i.e., Definition~\ref{ri_eq_notation} and Notation~\ref{not1}) and notations $\Sigma_0$ and $\Sigma_u$ in the proof of Theorem~\ref{thm2} (i.e., see the statement of Lemma~\ref{lemma0} and Notation~\ref{not4}).   We also prepare Notation~\ref{not2}.  
\begin{notation}\label{not2}
If a knot type $K(D_P)$ has an alternating knot diagram $D_P$ obtained by adding over/under information to $P$, the knot type is denoted by $K^{alt}(P)$.    
\end{notation}
  
By using Fact~\ref{AKthm}, we have Lemma~\ref{lemma3}.    
\begin{lemma}\label{lemma3}
Let $P$ be a knot projection.   Let $K^{alt}(P)$ be as in Notation~\ref{not2}.  

\noindent$(1)$ If $P \in \langle {\mathcal{T}} \rangle$, then $C(K^{alt}(P))$ $=$ $u^-(P)$ $= 1$.

\noindent$(2)$ If $P$ $\in \langle \mathcal{R} \rangle \cup \langle {\mathcal{P}} \rangle \cup \langle {\mathcal{T}} \rangle \sharp \langle {\mathcal{T}} \rangle$, then $C(K^{alt}(P))$ $=$ $u^-(P)$ $= 2$.    
\end{lemma}
\begin{proof}
Note that the minimal genus algorithm of Fact~\ref{AKthm} gives a surface $\Sigma_0$ that spans $K^{alt}(P)$ and has the maximal Euler characteristic $\chi(\Sigma_0)$.  

Suppose that 
$P \in \langle {\mathcal{T}} \rangle$.  Then, the set of alternating knot diagrams obtained from $P$ is fixed.   
Note that a state surface $\Sigma_u$ obtained from the computation of $u^-(P)$ is one of the minimal genus algorithm of Fact~\ref{AKthm} giving $\Sigma_0$.   Then, $\chi(\Sigma_0) = \chi(\Sigma_u)$.   By Lemma~\ref{lemma0}, 
$C(K^{alt}(P))=u^-(P)$.   Further, by Theorem~\ref{thm1}, $u^-(P)=1$.  
Then, we have (1).  

By replacing the assumption $P \in \langle {\mathcal{T}} \rangle$ with
\begin{equation*}\label{condition_2}
P \in \langle \mathcal{R} \rangle \cup \langle {\mathcal{P}} \rangle \cup \langle {\mathcal{T}} \rangle \sharp \langle {\mathcal{T}} \rangle,    
\end{equation*}
and by the same argument, we have (2).  
\end{proof}

\section{Alternating knots with crosscap number one revisited}\label{sec5}
As an application of Theorems~\ref{thm1} and  \ref{thm2}, it gives an elementary proof of a known result that for any alternating knot $K$, $C(K)=1$ if and only if $K$ is a $(2, 2l-1)$-torus knots ($l \ge 2$), as shown in Proposition~\ref{fact2}.    Before proving Proposition~\ref{fact2}, we need preliminary results.  
Note that Adams and Kindred obtain \cite[Corollary~6.1]{AK}.   Here, we use an expression \cite[Theorem~3.3]{KL} of \cite[Corollary~6.1]{AK}.  Note also Fact~\ref{w_fact}.           
\begin{fact}[an expression of Corollary~6.1 of \cite{AK}]\label{factKL}
Let $K$ be an alternating knot, $C(K)$ the crosscap number of $K$, and $g(K)$ the orientable genus of $K$.  
Let $Y$ be the set of state surfaces with maximal Euler characteristics obtained from the minimal genus algorithm as in Fact~\ref{AKthm}.  \\
Then, 
\begin{enumerate}
\item If there exists $\Sigma$ $(\in Y)$ that is a non-orientable, then $C(K)$ $=$ $1-\chi(\Sigma)$.   \label{case1}
\item If every $\Sigma$ $(\in Y)$ is orientable, then $C(K)$ $=$ $2-\chi(\Sigma)$ and $C(K)$ $=$ $2g(K)+1$.  \label{case2}
\end{enumerate}
\end{fact}
We also prepare the following technical lemma. 
\begin{lemma}\label{lemma4}
Let $P$ be a knot projection that is the image of a generic immersion $f: S^1$ $\to$ $S^2$.      
For every pair of two double points $d, d'$ of $P$, the configuration of $f^{-1}(\{ d, d' \})$ on $\subset S^1$ is one of two types $(a)$ and $(b)$ on $S^1$.   
In other words, any pair of two double points are represented by Figure~\ref{connection3}~$(a)$ or $(b)$ where dotted curves indicate the connections of double points.  
\begin{figure}[h!]
\includegraphics[width=8cm]{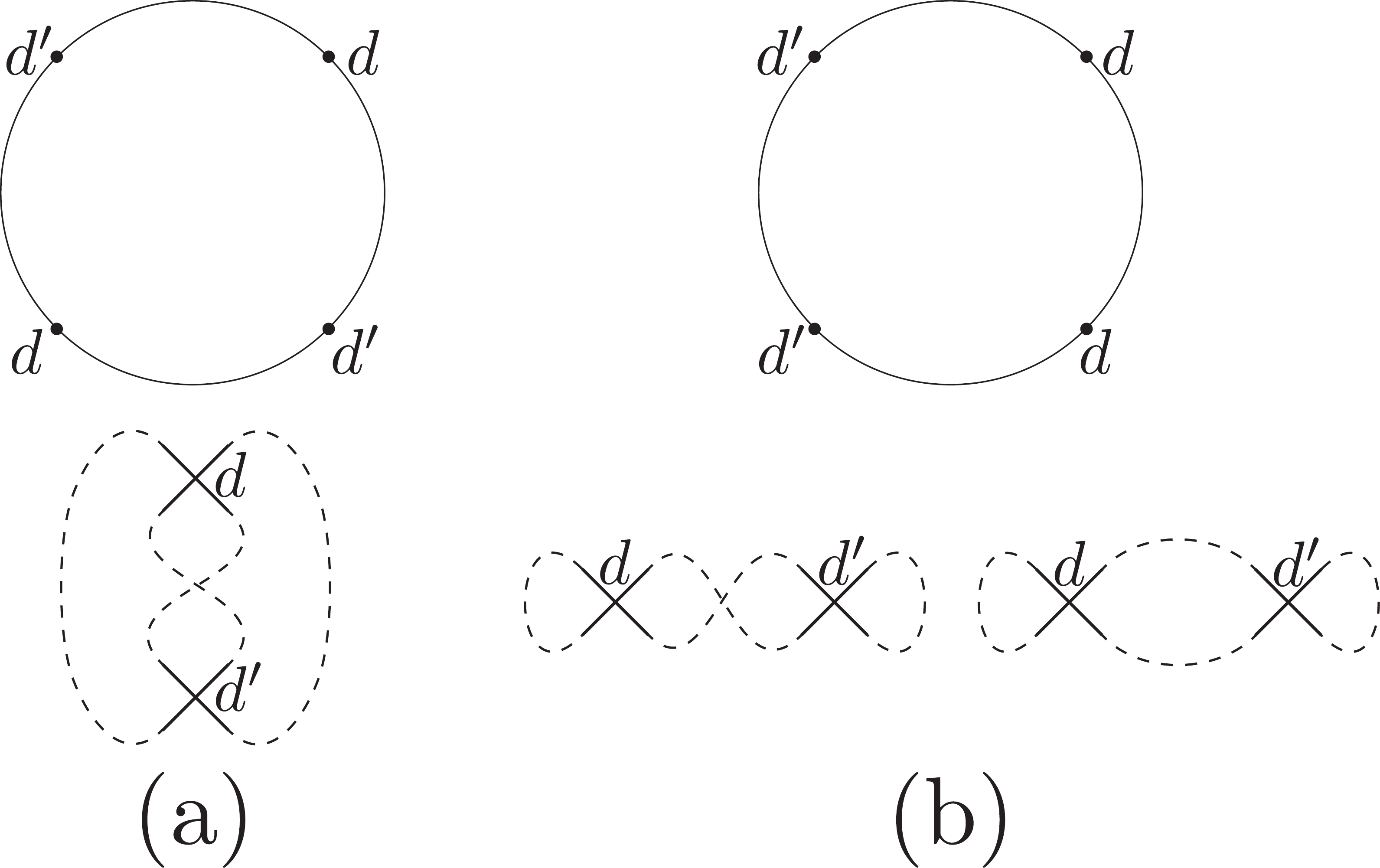}
\caption{In the upper line, the configuration of preimages of double points $d$ and $d'$.    
In the lower line, (a) : the leftmost knot projection and (b) : the two knot projections in the right half.  
Two double points and their connections.  Dotted curves indicate the connections of double points.}
\label{connection3}
\end{figure}
\end{lemma}
\begin{proof}
Every knot projection is a $1$-component curve, and thus, the possibilities of connections are shown in Figure~\ref{connection3}.   
\end{proof}
\begin{lemma}\label{connection}
If there exist two double points as in Figure~\ref{connection3}~$(a)$, then, after a Seifert splice at one of the two double point, any splice at the other double point yields another knot projection.      
\end{lemma}
\begin{proof}
By Lemma~\ref{lemma4} and Figure~\ref{connection4}, it is easy to see the claim.  
\end{proof}
\begin{figure}[h!]
\includegraphics[width=8cm]{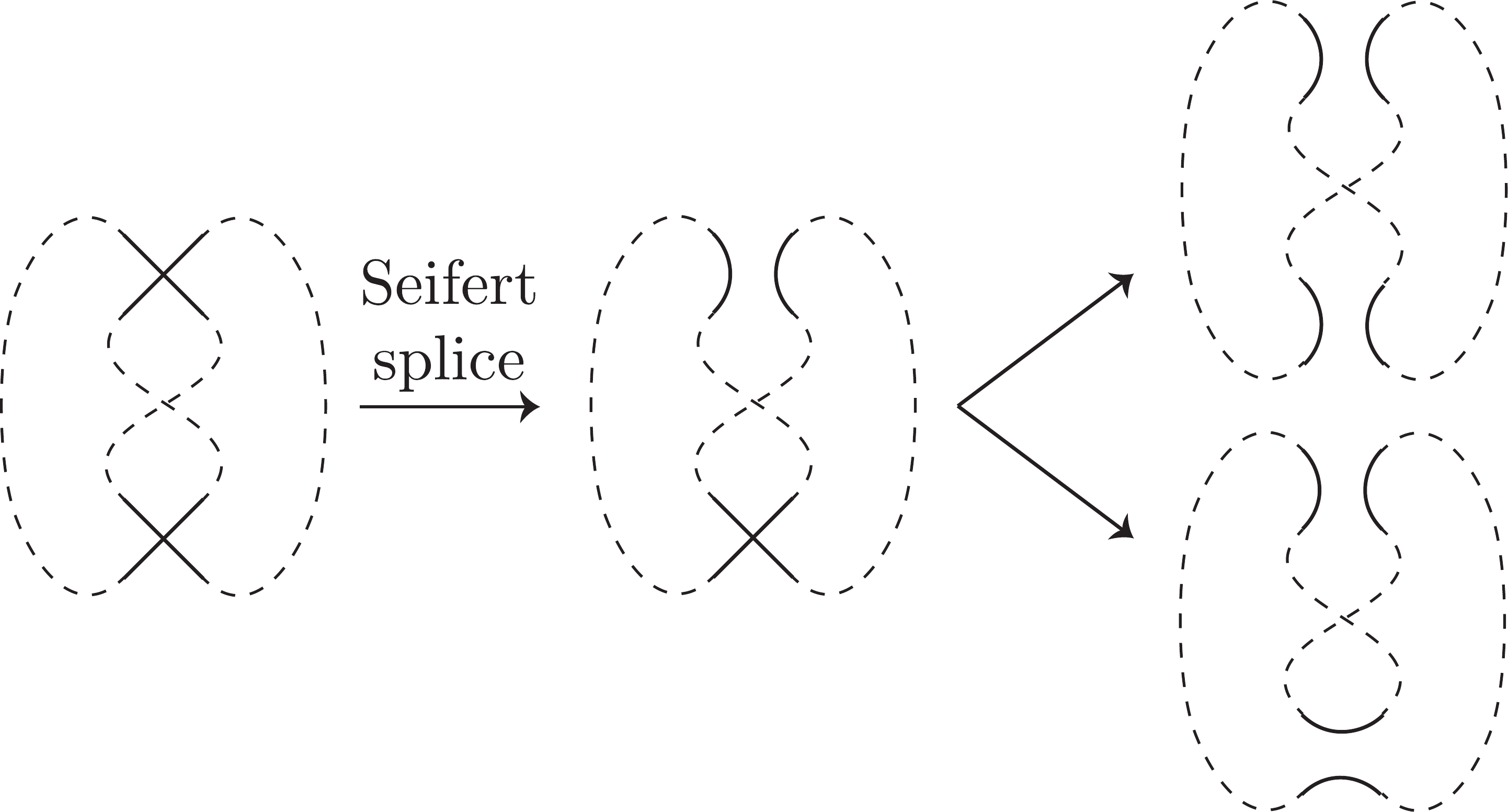}
\caption{Two Seifert splices on the two double points (upper arrow) and one Seifert splice and the other splice on the two double points (lower arrow)}\label{connection4}
\end{figure}
\begin{lemma}\label{eq4}
\begin{align*}
\langle {\mathcal{T}} \rangle
= 
\{P~|~C(K^{alt}(P)) = u^-(P)=1~(\forall P)\}.    
\end{align*}
\end{lemma}
\begin{proof}
For a knot projection $P$, let $K^{alt}(P)$ be a knot type as in Notation~\ref{not2}.    Then, 
\begin{align*}
\langle {\mathcal{T}} \rangle
&\stackrel{\textrm{Lemma~\ref{lemma3}~(1)}}{\subset}
\{P~|~C(K^{alt}(P)) = u^-(P)=1~(\forall P)\}\\
& \stackrel{\textrm{Theorem~\ref{thm1}~(1)}}{\subset} \langle {\mathcal{T}} \rangle.  
\end{align*}
\end{proof}
\begin{lemma}\label{eq5}  
\begin{align*}
{\mathcal{T}_{\kk}} = \{ K~{\textrm{: an alternating knot}}~|~C(K)=1, u^-(K)=1 \}.   
\end{align*}
\end{lemma}
\begin{proof}
Note that $P$ uniquely determines an alternating knot diagram (up to reflection).  For Lemma~\ref{eq4}, the left-hand side $\langle {\mathcal{T}} \rangle$  determines $\{ K^{alt}(P)~|~P \in \langle {\mathcal{T}} \rangle \}$, which equals $ \mathcal{T}_{\kk}$.   On the other hand, the right-hand side $\{P~|~C(K^{alt}(P))$ $=$ $u^-(P)=1~(\forall P)\}$ determines $\{K^{alt}(P)$ $|$ $C(K^{alt}(P))$ $=$ $u^-(P)=1$ $(\forall P)\}$, which equals $\{ K :$ an alternating knot $|~C(K)=1,$ $u^-(K)=1 \}$ (cf.~Definition~\ref{def_uk}).
\end{proof}
\begin{proposition}\label{fact2}
Let $\mathcal{T}_{\kk}$ be the set as in Notation~\ref{not3}.  
Let $K$ be an alternating knot and $C(K)$ the crosscap number of $K$.  Let $u^- (K)$ be the integer as in Definition~\ref{def_uk}.  
Then, the following conditions are mutually equivalent.  

\noindent$(A)$ $K$ $\in \mathcal{T}_{\kk}$.


\noindent $(B)$ $C(K)=1$.  

\noindent $(C)$ $u^-(K)  =1$.
\end{proposition}
\begin{proof}
\noindent (Proof of (A) $\Leftrightarrow$ (B).)   
Lemma~\ref{eq5} immediately implies that (A) $\Rightarrow$ (B).  

\noindent((B) $\Rightarrow$ (A).) 
Suppose that $C(K)=1$ and $K$ is an alternating knot. By definition, there exist an alternating knot diagram $D^{alt}(K)$ of $K$.  Let $P$ be a knot projection obtained from $D^{alt}(K)$ by ignoring over/under information of the double points.  
By Fact~\ref{factKL}, we have the following Case~1 and Case~2 corresponding to (\ref{case1}) and (\ref{case2}) of Fact~\ref{factKL}, respectively.    

\noindent $\bullet$ Case~1: there exist an alternating knot diagram $D^{alt}(K)$ of $K$, a state $s$ of $D^{alt}(K)$ such that a non-orientable state surface $\Sigma_0$ obtained from $D^{alt}(K)$ satisfies that  $C(K)=1-\chi(\Sigma_0)$, i.e., $\chi(\Sigma_0)$ $=$ $0$.  Here, note that the state $s$ is given by the algorithm of Fact~\ref{AKthm}.  
Let $n(P)$ be the number of double points of $P$.  The state $s$ is obtained from $P$ by $n(P)$ splices.  In the following, we find the state $s$ in the $2^{n(P)}$ candidates.  
Then, note that the splices consist of $n(P)-1$ Seifert splices producing $n(P)$-component curves and a single $S^-$ since $\chi(\Sigma_0)=0$.  Here, note that if there exist two splices of type $S^-$ in the $n(P)$ splices, then the $n(P)$ splices do not realize $\Sigma_0$ because $\chi(\Sigma_0)$ $=$ $1-C(K)$.   
Then, we interpret the $n(P)$ splices as a sequence of the $n(P)$ splices, and we may suppose that there exists a sequence such that
\[
P=P_0 \stackrel{Op_1}{\to} P_1 \stackrel{Op_2}{\to} P_2 \stackrel{Op_3}{\to} \dots \stackrel{Op_{n(P)}}{\to} P_{n(P)} = s  
\]
and $Op_i = S^-$ ($1 \le i \le n(P)$).  
By Lemma~\ref{connection}, for the double points corresponding to $Op_k$ $(1 \le k \le i-1)$, any two double points are represented as in Figure~\ref{connection3}~(b).  Here, if there exists a pair of type (a), then a pair consisting of two splices containing a Seifert splice on the two double points sends a $1$-component curve to another $1$-component curve, which implies the contradiction with the condition that $n(P)-1$ Seifert splices produce $n(P)$-component curves.  
Similarly, by Lemma~\ref{connection}, for two double points corresponding to $Op_k$ $(1 \le k \le i-1)$ and $Op_j$ $(i+1 \le j \le n(P))$, any pair is also represented as in Figure~\ref{connection3}~(b).  
Thus, noting that the state $s$ has one to one correspondence with the $n(P)$ splices (Definition~\ref{state}), it is easy to choose $\sigma_1,$ $\sigma_2, \dots, \sigma_{i-1}$ ($\sigma_{i+1},$ $\sigma_{i+2}, \dots, \sigma_{n(P)}$,~resp.)  like $\Sigma_u$ (Notation~\ref{not4})  corresponding to $\ri^-$'s applied successively to $P$ ($P_{i}$,~resp.) to obtain $P_{i-1}$ ($s$,~resp.).  Here, by Lemma~\ref{connection}, note that two double points corresponding to $\sigma_{k}$ $(1 \le k \le i-1)$ and $S^-$ are the configuration of  type~(b) of  Figure~\ref{connection3}.      
Then, $u^-(P) \le 1$.  Here, $K$ is not the unknot, $1 = C(K) \le u^-(P)$ ($\because$ Theorem~\ref{thm2}).  Thus, $C(K)=1$ and $u^- (P)=1$, where $P$ is a knot projection obtained from $D^{alt}(K)$.  Then, by Lemma~\ref{eq5}, we have $K \in {\mathcal{T}}_{\kk}$, which implies (A).   

\noindent $\bullet$ Case~2: For an orientable genus $g(K)$, $C(K)$ $=$ $2g(K)+1$.  If $C(K)=1$, then $g(K)=0$.  Then, $K$ is the unknot, which implies a contradiction.  

\noindent((A) $\Leftrightarrow$ (C).)
Since Lemma~\ref{eq5} immediately implies that (A) $\Rightarrow$ (C), it is sufficient to shown that (C) $\Rightarrow$ (A).  
Recall that $u^- (K)$ $=$ $\min_{P \in Z(K)} u^-(P)$ (Definition~\ref{def_uk}).  Then, 
\begin{align*}
& u^-(K)  =1 \\
\Rightarrow& \exists P \in Z(K)~{\textrm{such that}}~P \in \langle {\mathcal{T}} \rangle \quad(\because~{\rm{Theorem~\ref{thm1}~(1)}})\\
\Rightarrow& K \in {\mathcal{T}_{\kk}}.
\end{align*}   
\end{proof}
\section{Alternating knots with crosscap number two}\label{sec6}
By Theorems~\ref{thm1} and \ref{thm2}, we determine alternating knots with crosscap number two (Theorem~\ref{corollary1}).  
\begin{lemma}\label{eq6}
\begin{align*}
\langle {\mathcal{R}} \rangle \cup \langle {\mathcal{P}} \rangle \cup \langle {\mathcal{T}} \rangle \sharp \langle {\mathcal{T}} \rangle  
= 
\{P~|~C(K^{alt}(P)) = u^-(P)=2~(\forall P)\}.  
\end{align*}
\end{lemma}
\begin{proof}
For a knot projection $P$, let $K^{alt}(P)$ be as in Notation~\ref{not2}.  Then, 
\begin{align*}
&\langle {\mathcal{R}} \rangle \cup \langle {\mathcal{P}} \rangle \cup \langle {\mathcal{T}} \rangle \sharp \langle {\mathcal{T}} \rangle \\
\stackrel{\textrm{Lemma~\ref{lemma3}~(2)}}{\subset}&
\{P~|~C(K^{alt}(P)) = u^-(P)=2~(\forall P)\}\\
\stackrel{\textrm{Theorem~\ref{thm1}~(2)}}{\subset}&
\{ P~|~P \in \langle {\mathcal{R}} \rangle \cup \langle {\mathcal{P}} \rangle \cup \langle {\mathcal{T}} \rangle \sharp \langle {\mathcal{T}} \rangle \}.  
\end{align*}
\end{proof}
\begin{lemma}\label{eq7}
\begin{align*}
\mathcal{R}_{\kk} \cup {\mathcal{P}}_{\kk} \cup \mathcal{T}_{\kk} \sharp \mathcal{T}_{\kk} =  \{ K~{\textrm{: an alternating knot}}~|~C(K)=2, u^-(K)=2 \}.     
\end{align*}
\end{lemma}
\begin{proof}
Note that $P$ uniquely determines an alternating knot diagram (up to reflection).  
For Lemma~\ref{eq6}, $\{ P~|~P \in \langle {\mathcal{R}} \rangle \cup \langle {\mathcal{P}} \rangle \cup \langle {\mathcal{T}} \rangle \sharp \langle {\mathcal{T}} \rangle \}$ determines $\{ K^{alt}(P)$ $|$ $P \in \langle {\mathcal{R}} \rangle \cup \langle {\mathcal{P}} \rangle \cup \langle {\mathcal{T}} \rangle \sharp \langle {\mathcal{T}} \rangle \}$, which equals $\mathcal{R}_{\kk}$ $\cup$ ${\mathcal{P}}_{\kk}$ $\cup~ \mathcal{T}_{\kk} \sharp \mathcal{T}_{\kk}$.   Similarly, $\{P~|~C(K^{alt}(P))$ $=$ $u^-(P)=2~(\forall P)\}$ determines $\{K^{alt}(P)~|~C(K^{alt}(P))$ $=$ $u^-(P)=2~(\forall P)\}$, which equals $\{ K$ an alternating knot $|~C(K)=2,$ $u^-(K)=2 \}$ (cf.~Definition~\ref{def_uk}).  
\end{proof}
\begin{theorem}\label{corollary1}
Let $\mathcal{R}_{\kk}$, $\mathcal{P}_{\kk}$, and $\mathcal{T}_{\kk} \sharp \mathcal{T}_{\kk}$ be as in Notation~\ref{not3}.       
Let $K$ be an alternating knot and $C(K)$ the crosscap number of $K$.  Let $u^- (K)$ be an integer as in Definition~\ref{def_uk}.    
Then, the following conditions are mutually equivalent.  
 
\noindent $(A)$ $K$ $\in \mathcal{R}_{\kk} \cup {\mathcal{P}}_{\kk}$ $\cup~\mathcal{T}_{\kk} \sharp \mathcal{T}_{\kk}$.


\noindent $(B)$ $C(K)$ $=$ $2$.  

\noindent $(C)$ $u^- (K)  =2$.
\end{theorem}
\begin{proof}
\noindent (Proof of (A) $\Leftrightarrow$ (B).) 
Lemma~\ref{eq7} immediately implies that (A) $\Rightarrow$ (B).  

\noindent((B) $\Rightarrow$ (A).)   
Suppose that $C(K)=2$ and $K$ is an alternating knot.  
By definition, there exist an alternating knot diagram $D^{alt}(K)$ of $K$.  Let $P$ be a knot projection obtained from $D^{alt}(K)$ by ignoring over/under information of the double points.  
By Fact~\ref{factKL}, we have the following Case~1 and Case~2 corresponding to (\ref{case1}) and (\ref{case2}) of Fact~\ref{factKL}, respectively.  

\noindent Case~1: there exist an alternating knot diagram $D^{alt}(K)$ of $K$, a state $s$ of $D^{alt}(K)$ such that a non-orientable state surface $\Sigma_0$ obtained from $D^{alt}(K)$ satisfies that $C(K)=1-\chi(\Sigma_0)$, i.e., $\chi(\Sigma_0)$ $=$ $-1$.  Here, note that the state $s$ is given by the algorithm of Fact~\ref{AKthm}.  Let $n(P)$ be the number of double points of $P$.  The state $s$ is obtained from $P$ by $n(P)$ splices.  In the following, we find the state $s$ in the $2^{n(P)}$ candidates.  If the $n(P)$ splices are $n(P)$ Seifert splices, they give an orientable surface, which implies a contradiction.  Thus, there exists at least one $S^-$ in the $n(P)$ splices.  
Further, since $\chi(\Sigma_0)= -1$, the splices consist of $n(P) -2$ Seifert splices produce an $n(P) -1$-component curve and exactly two $S^-$'s (there are no other possibilities).     
Then, we interpret the $n(P)$ splices as a sequence of the $n(P)$ splices, and suppose that there exists a sequence such that
\[
P=P_0 \stackrel{Op_1}{\to} P_1 \stackrel{Op_2}{\to} P_2 \stackrel{Op_3}{\to} \dots \stackrel{Op_{n(P)}}{\to} P_{n(P)} = s,   
\]
$Op_i = S^-$ and $Op_j = S^-$ ($1 \le i < j \le n(P)$).   
By Lemma~\ref{connection}, for the two distinct double points corresponding to $Op_k$ $(1 \le k \le i-1)$ and $Op_t$ $(1 \le t \le n(P), k \neq t)$, any two double points are represented as in Figure~\ref{connection3}~(b).  Here, if there exists a pair of type (a), then a pair consisting of two splices on the two double points sends a $1$-component curve to another $1$-component curve, which implies the contradiction with the condition that $n(P)-2$ Seifert splices produce $n(P)-1$-component curves.
Thus, noting that the state $s$ has one to one correspondence with the $n(P)$ splices (Definition~\ref{state}), we can choose $\sigma_1, \sigma_2, \ldots, \sigma_{i-1}$  like $\Sigma_u$ (Notation~\ref{not4})  corresponding to $\ri^-$'s applied successively to $P$ to obtain $P_{i-1}$.  Here, by Lemma~\ref{connection}, note that two double points corresponding to $\sigma_k$ $(1 \le k \le i-1)$ and $Op_i (= S^-)$ are the configuration of  type~(b) of  Figure~\ref{connection3}.     After applying $S^-$, a sequence consisting of a single $S^-$ and $j-i-1$ Seifert splices from $P_i$ to $P_j$, where $j-i-1$ Seifert splices should produce $j-i-1$ new components.  By recalling Definition~\ref{state}, the state $s$ has one to one correspondence with the $n(P)$ splices.  Then, 
by focusing $1$-gons, it is easy to choose $\sigma_{i+1}, \sigma_{i+2}, \ldots, \sigma_{j-1}$  like $\Sigma_u$ (Notation~\ref{not4})  corresponding to $\ri^-$'s applied successively to $P_i$ to obtain the state $P_{j-1}$.  Here, by Lemma~\ref{connection},  note that two double points corresponding to $\sigma_{k'}$ $(i+1 \le k' \le j-1)$ and $Op_j (= S^-)$ are the configuration of  type~(b) of  Figure~\ref{connection3}.  
Similarly, it is elementary to choose $\sigma_{j+1}, \sigma_{j+2}, \ldots, \sigma_{n(P)}$ like $\Sigma_u$ (Notation~\ref{not4}) corresponding to $\ri^-$'s applied successively to $P_j$ to obtain the state $s$.  
Then, $u^-(P) \le 2$.  Here, $K$ is not the unknot and is not in $\mathcal{T}_{\kk}$, $2 \le C(K) \le u^-(P)$ ($\because$ Theorem~\ref{thm2}).  Thus, $C(K)=2$ and $u^-(P) =2$, where $P$ is a knot projection obtained from $D^{alt}(K)$.  Then, by Lemma~\ref{eq7}, we have $K \in \mathcal{R}_{\kk}$ $\cup$ ${\mathcal{P}}_{\kk}$ $\cup$ $\mathcal{T}_{\kk} \sharp \mathcal{T}_{\kk}$, which implies (A).

\noindent Case~2: For an orientable genus $g(K)$, $C(K)$ $=$ $2g(K)+1$, which implies a contradiction with $C(K)=2$, which is an even number. 

\noindent((A) $\Leftrightarrow$ (C).)

Since Lemma~\ref{eq7} immediately implies that (A) $\Rightarrow$ (C), it is sufficient to shown that (C) $\Rightarrow$ (A).  
Recall that $u^- (K)$ $=$ $\min_{P \in Z(K)} u^-(P)$ (Definition~\ref{def_uk}).  Then, 
\begin{align*}
&u^- (K) =2\\
\Rightarrow& \exists P \in Z(K)~{\textrm{such that}}~P \in \langle {\mathcal{R}} \rangle \cup \langle {\mathcal{P}} \rangle \cup \langle {\mathcal{T}} \rangle \sharp \langle {\mathcal{T}} \rangle \quad(\because~{\rm{Theorem~\ref{thm1}~(2)}}) \\
\Rightarrow& K \in \mathcal{R}_{\kk} \cup {\mathcal{P}}_{\kk} \cup \mathcal{T}_{\kk} \sharp \mathcal{T}_{\kk}.
\end{align*}  
\end{proof}

\section{Additivity of $u^- (P)$}\label{sec7} 
In this section, we freely use notations in Definition~\ref{dfn_connected}.  
\begin{proposition}\label{prop1}
Let $P_1$ and $P_2$ be knot projections.  
\[u^- ( P_1 \sharp P_2 ) = u^-(P_1) + u^-(P_2).
\]
\end{proposition}    
\begin{proof}
Let $P$ $=$ $P_1 \sharp P_2$.  
Note that by definition, $u^-(P_1 \sharp P_2)$ $\le$ $u^-(P_1)$ $+$ $u^-(P_2)$.  
For any orientation, every $S^-$ is characterized by local oriented arcs, as shown in Figure~\ref{orientedS}.  On the other hand, when we choose appropriate orientations of $P_1$ and $P_2$, every connected sum $P_1 \sharp P_2$ does not change orientations of factors $P_1$ and $P_2$, as shown in Figure~\ref{orientedConn}.  Therefore, type $S^-$ ($\ri^-$,~resp.) on $P_i$ ($i=1, 2$) one to one corresponds to that of $P_1 \sharp P_2$, which implies that    
$u^-(P_1 \sharp P_2)$ $\ge$ $u^-(P_1)$ $+$ $u^-(P_2)$.  
\end{proof}
\begin{figure}[h!]
\includegraphics[width=8cm]{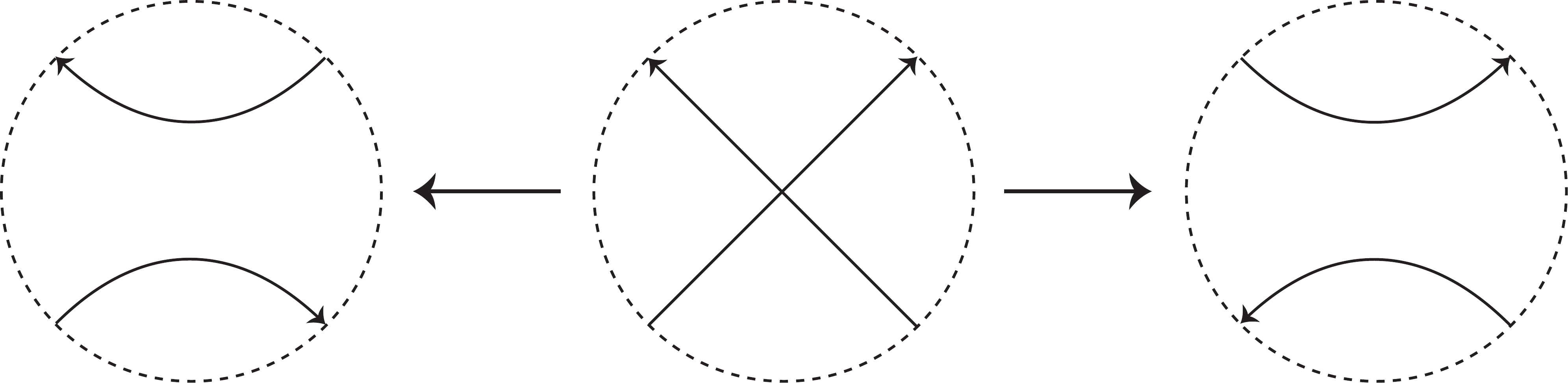}
\caption{Every $S^-$ is characterized by local oriented arcs.  }\label{orientedS}
\end{figure}
\begin{figure}[h!]
\includegraphics[width=10cm]{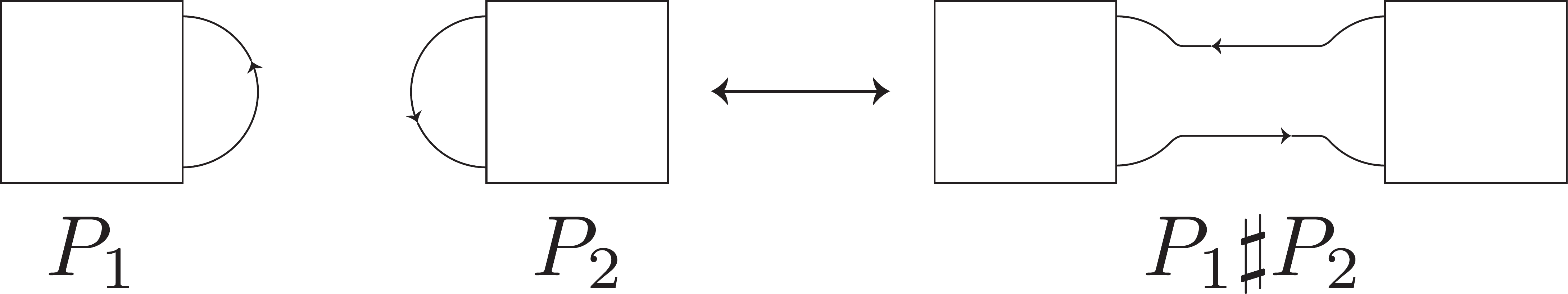}
\caption{An operation $\sharp$ preserves  orientations of $P_1$ and $P_2$.}\label{orientedConn}
\end{figure}
As a corollary of Proposition~\ref{prop1}, we have Corollary~\ref{cor3} (cf.~Theorem~\ref{thm2}).   
\begin{corollary}\label{cor3}
For a knot projection $P$, $D_P$ and $K(D_P)$ be as in Definition~\ref{state}, and let $C(K(D_P))$ be the crosscap number of $K(D_P)$.  
Let $P_1$ and $P_2$ be knot projections.  
Suppose that $C(K(D_{P_1 \sharp P_2}))$ $\neq$ $C(K(D_{P_1}))$ $+$ $C(K(D_{P_2}))$.  Then, 
\[
C(K(D_{P_1 \sharp P_2})) < u^- (K(D_{P_1 \sharp P_2})).  
\]
\end{corollary}
Recall (cf.~Notation~\ref{notation_p}, Definition~\ref{ri_eq_notation}) that $S^+$ and $\ri^+$ are the respective inverse of $S^-$ and $\ri^-$.  
\begin{definition}[unknotting-type number $u(P)$]\label{dfn3_u}
Let $P$ be a knot projection and $O$ the simple closed curve.  The nonnegative integer $u(P)$ is defined as the minimum number of operations of types  $S^{\pm}$ to obtain $O$ from $P$ by a finite sequence of operations of types $S^{\pm}$ and of types $\ri^{\pm}$.  
\end{definition}
\begin{example}\label{74two}
In general, the crosscap number is not additive under the connected sum \cite{MY}.  For example, for the knot $7_4$, $C(7_4)$ $=$ $3$ and $C(7_4 \sharp 7_4)$ $=$ $5$.     For the knot projection $\widehat{7_4}$, by Theorem~\ref{thm1} and Proposition~\ref{prop1}, $u^-(\widehat{7_4})$ $=$ $3$ and $u^-(\widehat{7_4} \sharp \widehat{7_4})$ $=$ $u^-(\widehat{7_4})$ $+$ $u^-(\widehat{7_4})$ $=$ $3$ $+$ $3$ $=$ $6$.   
However, the behavior of the number $u(P)$, introduced in Definition~\ref{dfn3_u}, is different from that of $u^-(P)$.   By definition, $u(P) \le u^-(P)$ for every knot projection $P$.  
We have $u(\widehat{7_4})$ $\le$ $u^-(\widehat{7_4})$ $=$ $3$ and $u(\widehat{7_4} \sharp \widehat{7_4})$ $\le$ $5$, as shown in Figure~\ref{f15}.     
\end{example}
\begin{figure}[h!]
\includegraphics[width=12cm]{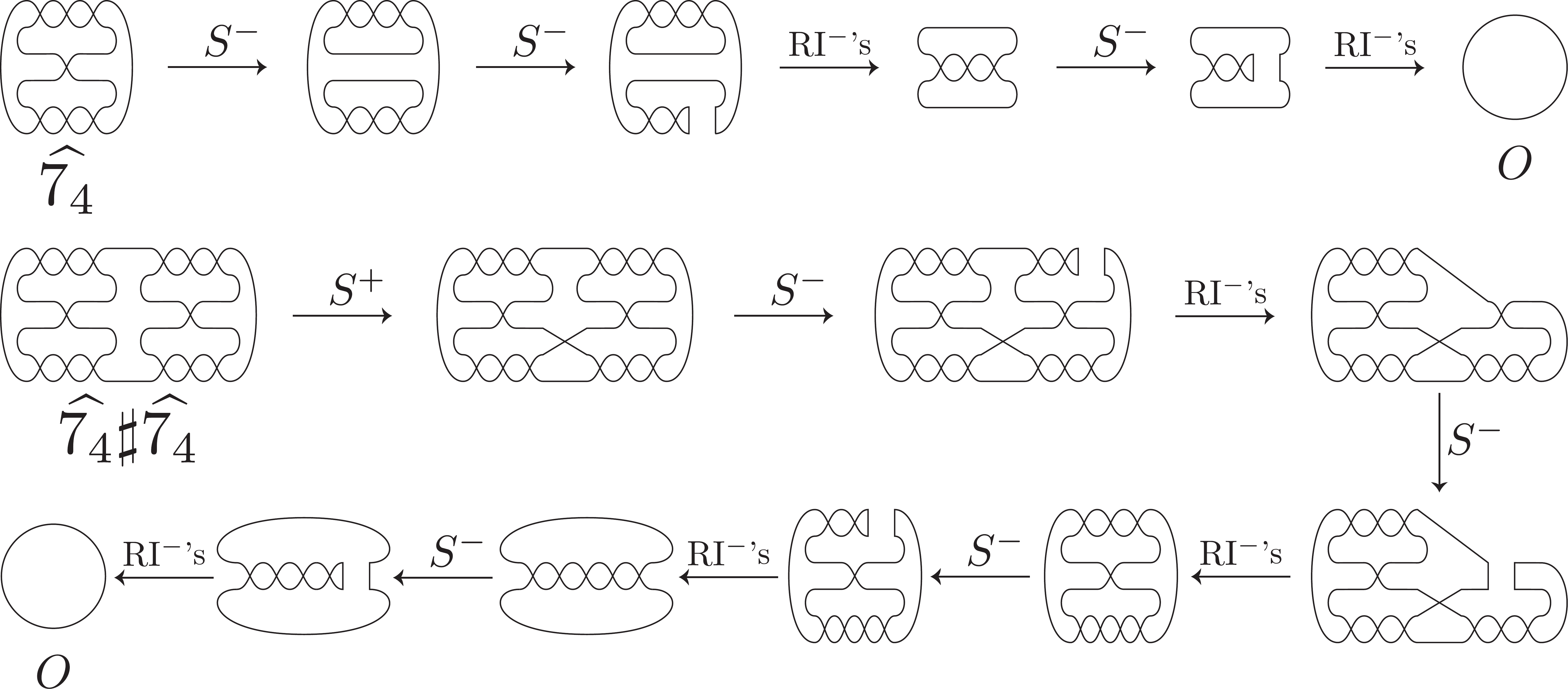}
\caption{$u(\widehat{7_4})$ $\le$ $u^-(\widehat{7_4})$ $=$ $3$, $u(\widehat{7_4} \sharp \widehat{7_4})$ $\le$ $5$.}\label{f15}
\end{figure}
For $u(P)$, hoping for the best, we ask the following question: 
\begin{question}
Let $P$ be a knot projection and $D_P$ a knot diagram by adding any over/under information to each double point of $P$.  Let $K(D_P)$ be a knot type having a knot diagram $D_P$.  
Let $C(K(D_P))$ be the crosscap number of $K(D_P)$.  
Then, does every knot projection $P$ hold 
\[  
C(K(D_P)) \le u (P)? 
\]
\end{question}
\begin{remark}
Let $Z(K)$ be the set of knot projections obtained from alternating knot diagrams of $K$.  Then, $\min_{P \in Z(K)} u(P)$ is an alternating knot invariant.  Let $u (K)$ $=$ $\min_{P \in Z(K)} u(P)$ (cf.~Definition~\ref{def_uk}).      
By Theorem~\ref{thm2}, for every knot $K$, $C(K) \le$ $u^- (K)$.  However, it is unknown whether $C(K) \le$ $u (K)$ or not.  
\end{remark}
\begin{example}\label{74general}
If we generalize Example~\ref{74two}, we have examples Case~1--Case~3, as shown in Figs.~\ref{74a}--\ref{74c} by using connected sums of knot projections where each component is a knot projection, as shown in Figure~\ref{f15a}.   
Namely, there exist infinitely many knot projections, each of which is represented as $P \sharp P'$ such that $g(K^{alt}(P))$ $=$ $1$, $g(K^{alt}(P'))$ $=$ $1$, $C(K^{alt}(P))$ $=3$, $C(K^{alt}(P'))$ $=3$, $C(K^{alt}(P) \sharp K^{alt}(P'))$ $=5$, 
$u^-(P \sharp P')$ $=$ $6$, and $u(P \sharp P') \le 5$.  Here, for $C(K^{alt}(P) \sharp K^{alt}(P'))$ $=5$, we use Fact~\ref{AKthm}.  
For the initial knot projection $P \sharp P'$ in each figure of Figs.~\ref{74a}--\ref{74c}, each symbol, ``odd" or ``even", indicates the number of given double points.       
\end{example}
\begin{figure}[h!]
\includegraphics[width=8cm]{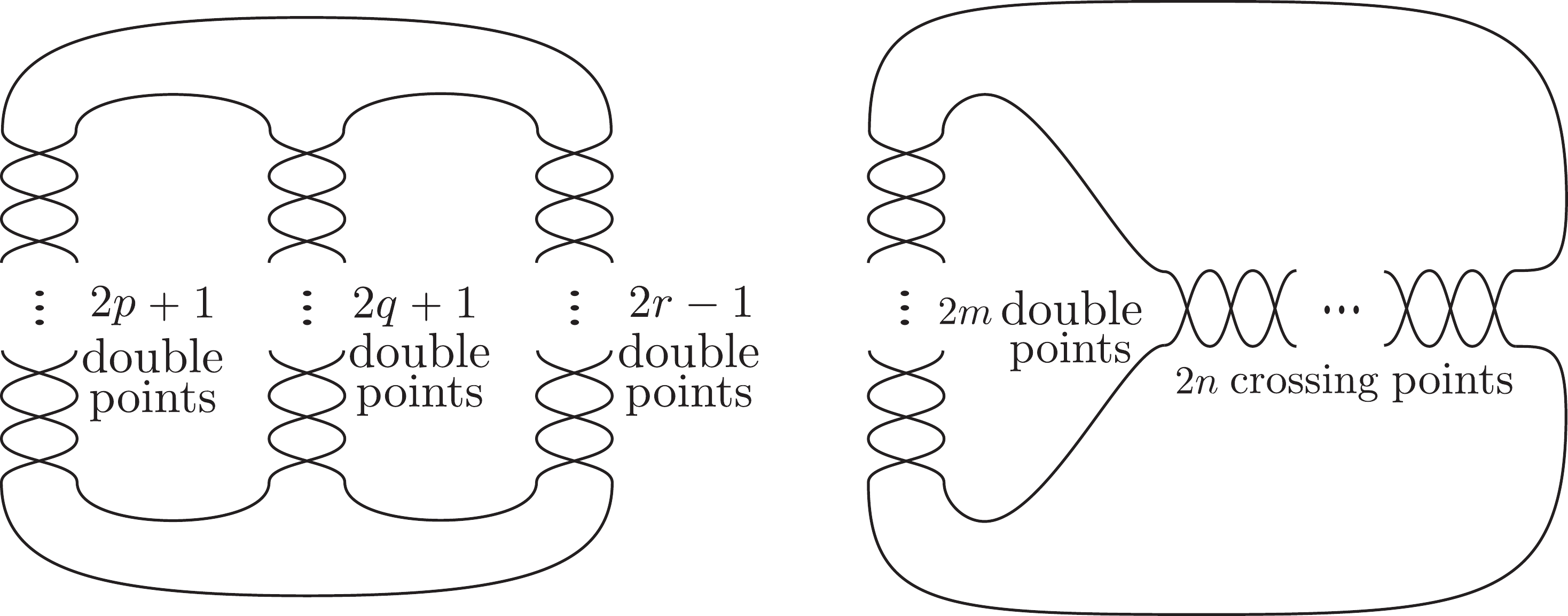}
\caption{$p, q, r \ge 1$ and $m, n \ge 2$. }\label{f15a}
\end{figure}
\begin{figure}[h!]
\includegraphics[width=12cm]{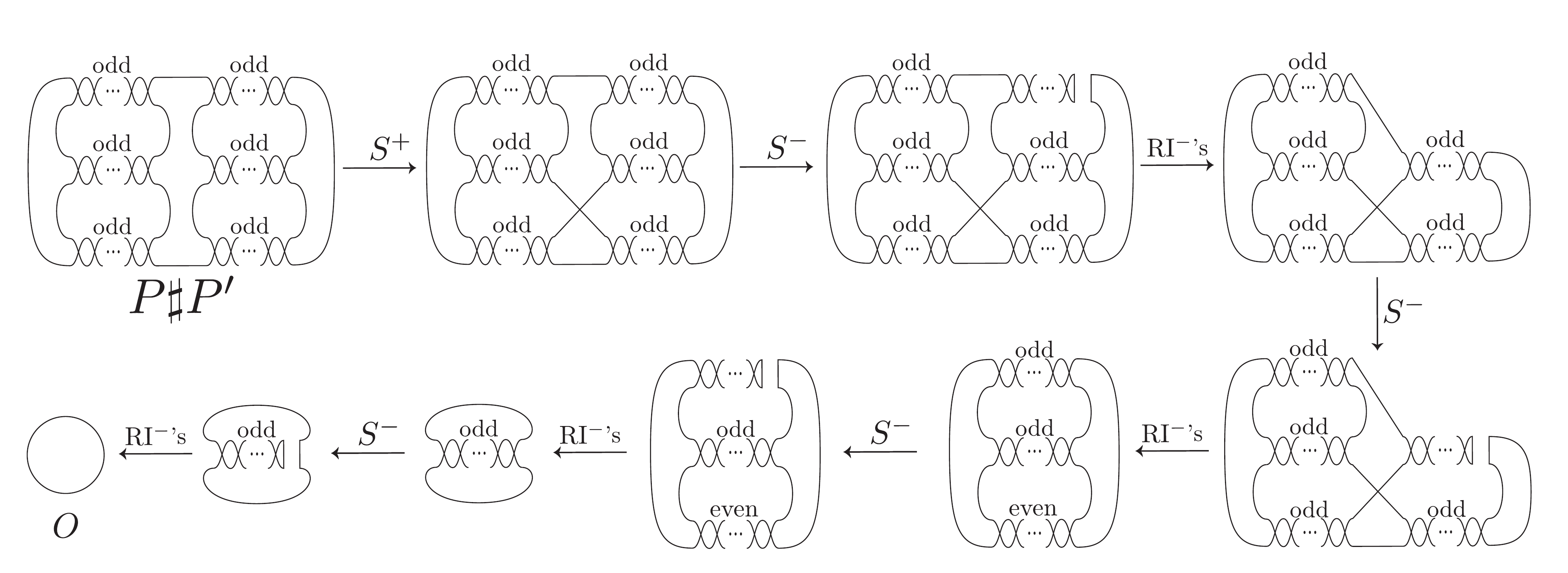}
\caption{Case~1}\label{74a}
\end{figure}
\begin{figure}[h!]
\includegraphics[width=12cm]{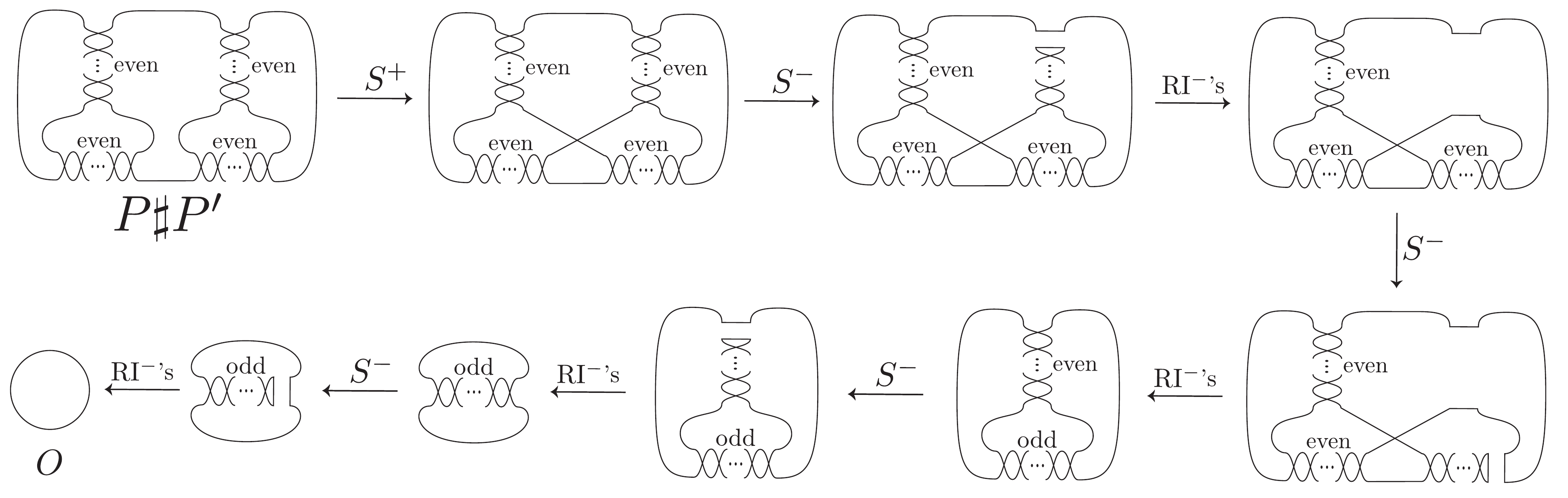}
\caption{Case~2}\label{74b}
\end{figure}
\begin{figure}[h!]
\includegraphics[width=12cm]{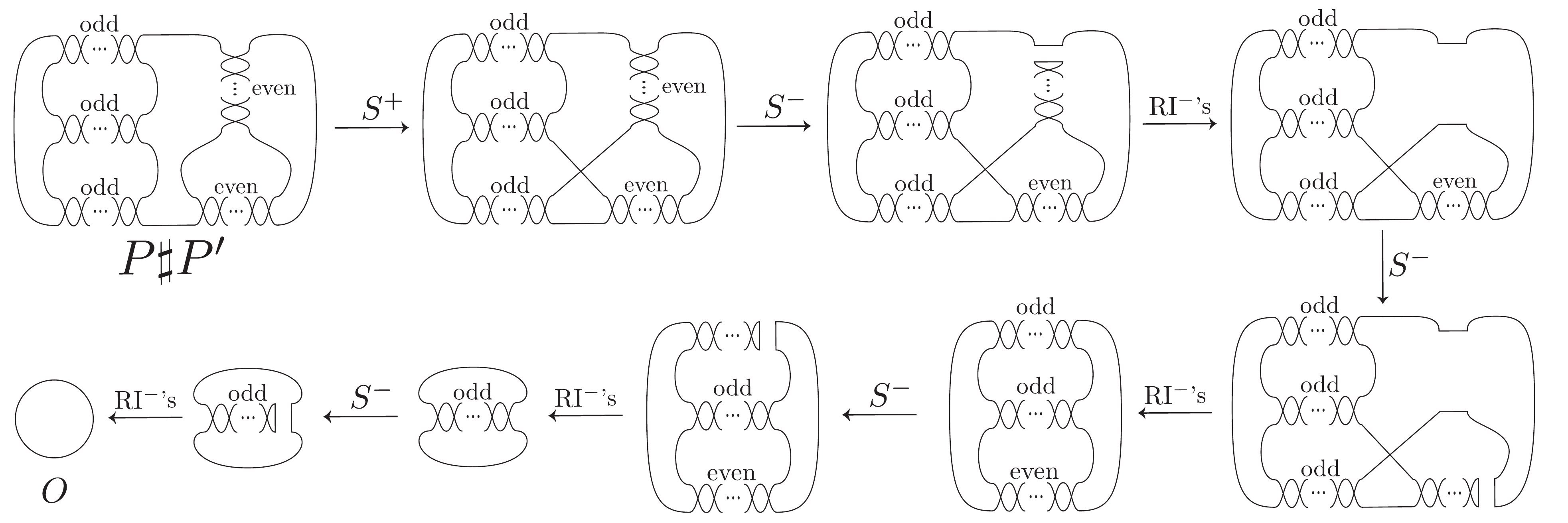}
\caption{Case~3}\label{74c}
\end{figure}
Finally, we remark Propositions~\ref{proposition1} and \ref{proposition2} and Questions~\ref{q2} and \ref{q3}.  
\begin{proposition}\label{proposition1}    
The following conditions are mutually equivalent.  

\noindent$(1)$ $P \in \langle {\mathcal{T}} \rangle$. 

\noindent$(2)$ $u^-(P)=1$.

\noindent$(3)$ $u(P)=1$.
\end{proposition}
\begin{proof}
By Theorem~\ref{thm1}, $(1)$ $\Leftrightarrow$ $(2)$.  
By the same argument as in Theorem~\ref{thm1}, it is easy to show that (1) $\Leftrightarrow$ (3).  
\end{proof}
\begin{proposition}\label{proposition2}  
The following conditions are mutually equivalent.  

\noindent$(1)$ $P$ $\in \langle \mathcal{R} \rangle \cup \langle {\mathcal{P}} \rangle \cup \langle {\mathcal{T}} \rangle \sharp \langle {\mathcal{T}} \rangle$. 

\noindent$(2)$ $u^-(P)=2$.

\noindent$(3)$ $u(P)=2$.
\end{proposition}
\begin{proof}
By Theorem~\ref{thm1}, $(1)$ $\Leftrightarrow$ $(2)$.  
By the same argument as in Theorem~\ref{thm1}, it is easy to show that (1) $\Leftrightarrow$ (3).  
\end{proof}
\begin{question}\label{q2}
Is there a prime knot knot projection $P$ such that 
$u(P) < u^- (P)$?
\end{question}
\begin{question}\label{q3}
Is there a prime knot knot projection $P$ such that 
$C(K^{alt}(P)) < u^- (P)$?
\end{question}

\section{Acknowledgement}
The authors would like to thank the referee for the comments.  
The authors would like to thank some participants of Topology Seminar at Tokyo Woman's Christian University for useful comments.  
The authors would like to thank Professor Tsuyoshi Kobayashi, Professor~Makoto Ozawa,  Professor Masakazu Teragaito, and Professor~Akira Yasuhara for their comments.   N.~I.~was partially supported by Sumitomo Foundation (Grant for Basic Science Research Projects, Project number: 160556).


\begin{thebibliography}{99}
\bibitem{AK} C.~Adams and T.~Kindred, A classification of spanning surfaces for alternating links, \emph{Algebr.\ Geom.\ Topol.} {\bf{13}} (2013), 2967--3007.  
\bibitem{BO} B.Burton and M.~Ozlen, Computing the crosscap number of a knot using integer programming and normal surfaces.  \emph{ACM Trans.\ Math.\ Software} {\bf{39}} (2012), Art.~4, 18pp.  
\bibitem{Ca} J.~Calvo, Knot enumeration through flypes and twisted splices, \emph{J.~Knot theory Ramifications} {\bf{6}} (1997), 785--798.  
\bibitem{CL} J.~C.~Cha and C.~Livingston, \emph{KnotInfo: Table of Knot Invariants}, \texttt{http://www.indiana.edu/ {\textasciitilde}knotinfo}, July 20, 2017.   
\bibitem{clark} B.~E.~Clark, Crosscaps and knots, \emph{Internat. J. Math. Sci.} {\bf{1}} (1978), 113--123.   
\bibitem{HaT} H.~Hatcher and W.~Thurston, Incompressible surfaces in $2$-bridge knot complements, \emph{Invent.\ Math.\ } {\bf{79}} (1985), 225--246.  
\bibitem{HT} M.~Hirasawa and M.~Teragaito, Crosscap numbers of $2$-bridge knots, \emph{Topology} {\bf{45}} (2006), 513--530.  
\bibitem{IM} K.~Ichihara and S.~Mizushima, Crosscap numbers of pretzel knots, \emph{Topology Appl.} {\bf{157}} (2010), 193--201.  
\bibitem{ItoShimizu} N.~Ito and A.~Shimizu, The half-twisted splice operation on reduced knot projections, \emph{J.~Knot Theory Ramifications} {\bf{21}} (2012), 1250112, 10 pp.
\bibitem{IT_triple1} N.~Ito and Y.~Takimura, Triple chords and strong (1, 2) homotopy, \emph{J.~Math.\ Soc.\ Japan} {\bf{68}} (2016), 637--651.   
\bibitem{IT8c} N.~Ito and Y.~Takimura, The tabulation of prime knot projections with their mirror images up to eight double points, \emph{Topology Proc.}, accepted (June 30, 2018).  
\bibitem{K} L.~H.~Kauffman, State models and the Jones polynomial, \emph{Topology} {\bf{26}} (1987), 395--407.  
\bibitem{KL} E.~Kalfagianni and C.~R.~S.~Lee, Crosscap numbers and the Jones polynomial.  \emph{Adv.\ Math.} {286} (2016), 308--337.  
\bibitem{Ke} A.~Kerian, Crosscap number: Handcuff graphs and unknotting number.  Thesis (Ph.D.)--The university of Nebraska - Lincoln.  \emph{ProQuest LLC, Ann Arbor, MI,} 2015, 53pp.
\bibitem{MY} H.~Murakami and A.~Yasuhara, Crosscap number of a knot, \emph{Pacific J. Math.} {\bf{171}} (1995), 261--273.   
\bibitem{Ro} D.~Rolfsen, \emph{Knots and links}.   Mathematical Lecture Series, No.~7.  \emph{Publish or Perish, Inc., Berkley, Calif.,} 1976.  
\bibitem{Tra} M.~Teregaito, Crosscap numbers of torus knots, \emph{Topology Appl.} {\bf{138}} (2004), 219--238.  
\end{thebibliography}
\end{document}